\documentclass{article}
\usepackage{amsmath, amssymb, amsthm, fullpage, color, graphicx}
\usepackage{algorithm, algorithmic, tikz}
\usetikzlibrary{calc}

\newtheorem{lemma}{Lemma}
\newtheorem{theorem}{Theorem}
\theoremstyle{remark}
\newtheorem{remark}{Remark}
\DeclareMathOperator{\poly}{poly}
\newcommand{\hatgamma}{{\widehat{\gamma}}}
\newcommand{\epcount}{\tau}
\newcommand{\roundctr}{i}
\newcommand{\xset}{\mathcal{X}}
\newcommand{\inrad}{r}
\newcommand{\outrad}{R}
\newcommand{\ball}{\mathbb{B}}
\newcommand{\centerpt}{\ensuremath{x_0}}
\newcommand{\apex}{\ensuremath{y}}
\newcommand{\wideangle}{\ensuremath{\varphi}}
\newcommand{\base}{\ensuremath{z}}
\renewcommand{\top}{\ensuremath{\mbox{\textsc{top}}}}
\newcommand{\bottom}{\ensuremath{\mbox{\textsc{bottom}}}}
\newcommand{\apexsc}{\ensuremath{\mbox{\textsc{apex}}}}
\newcommand{\order}{\mathcal{O}}
\newcommand{\pyracenter}{\ensuremath{\operatorname{\textsc{center}}}}
\newcommand{\wideanglefinal}{\ensuremath{\bar{\varphi}}}
\newcommand{\cone}{\ensuremath{\mathcal{K}}}
\newcommand{\pyramid}{\ensuremath{\Pi}}

\newcommand{\halving}{\ensuremath{\rho}}
\newcommand{\otil}{\ensuremath{\widetilde{\mathcal{O}}}}
\newcommand{\Ball}{\ensuremath{\mathcal{B}}}
\newcommand{\vol}{\ensuremath{\mbox{vol}}}
\newcommand{\unitvol}{\ensuremath{V}}
\newcommand{\Reals}{\mathbb{R}}
\newcommand{\UB}{\operatorname{UB}}
\newcommand{\LB}{\operatorname{LB}}
\newcommand{\noise}{\varepsilon}
\newcommand{\E}{{\mathbb{E}}}
\renewcommand{\P}{{\mathbb{P}}}
\newcommand{\event}{\ensuremath{\mathcal{E}}}
\newcommand{\inconst}{\ensuremath{c_1}}
\newcommand{\phiconst}{\ensuremath{c_2}}
\newenvironment{proof-of-theorem}[1][{}]{\noindent{\bf Proof of
    Theorem #1}
  \hspace*{1em}}{\qed\smallskip\\}
\newcommand{\otherdelta}{\ensuremath{\bar{\Delta}}}
\newcommand{\defeq}{\ensuremath{:=}}

\begin{document}

\begin{center}
  {\LARGE{{\bf{Stochastic convex optimization with bandit feedback}}}} 

\vspace*{.2in}

\begin{tabular}{ccccc}
  Alekh Agarwal$^\dagger$
  &
  Dean P. Foster$^\star$
  &
  Daniel Hsu$^\ddag$
  &
  Sham M. Kakade$^{\star,\ddag}$
  &
  Alexander Rakhlin$^\star$
\end{tabular}

\vspace*{.2in}

\begin{tabular}{ccc}
  Department of EECS$^\dagger$
  &   Department of Statistics$^\star$ & Microsoft Research$^\ddag$\\
  University of California, Berkeley & University of Pennsylvania &
  New England \\
  Berkeley, CA 94720 & Philadelphia, PA 19104 & Cambridge, MA 02142\\
\end{tabular}
\end{center}

\vspace*{1cm}

\begin{abstract}
  This paper addresses the problem of minimizing a convex, Lipschitz
  function $f$ over a convex, compact set $\xset$ under a stochastic
  bandit feedback model. In this model, the algorithm is allowed to
  observe noisy realizations of the function value $f(x)$ at any
  query point $x \in \xset$. The quantity of interest is the regret of
  the algorithm, which is the sum of the function values at
  algorithm's query points minus the optimal function value. We
  demonstrate a generalization of the ellipsoid algorithm that incurs
  $\otil(\poly(d)\sqrt{T})$ regret. Since any algorithm has regret at
  least $\Omega(\sqrt{T})$ on this problem, our algorithm is optimal
  in terms of the scaling with $T$.
\end{abstract}

\section{Introduction}
The classical multi-armed bandit problem, formulated by Robbins in
1952, is arguably the most basic setting of sequential decision-making
under uncertainty. Upon choosing one of $k$ available actions
(``arms''), the decision-maker observes an i.i.d.~realization of the
arm's cost drawn according to a distribution associated with the
arm. The performance of an allocation rule (algorithm) in sequentially
choosing the arms is measured by {\em regret}, that is the difference
between the expected costs of the chosen actions as compared to the
expected cost of the best action. Various extensions of the classical
formulation have received much attention in recent years. In
particular, research has focused on the development of optimal and
efficient algorithms for multi-armed bandits with large or even
infinite action spaces, relying on various assumptions on the
\emph{structure} of costs (rewards) over the action space. When such a
structure is present, the information about the cost of one arm
propagates to other arms as well, making the problem tractable. For
instance, the mean cost function is assumed to be linear in the
paper~\cite{DanHayKak08}, facilitating global ``sharing of
information'' over a compact convex set of actions in a
$d$-dimensional space. A Lipschitz condition on the mean cost function
allows a local propagation of information about the arms, as costs
cannot change rapidly in a neighborhood of an action. This has been
exploited in a number of works, notably
~\cite{Agr95,Kle05,KleSliUpf08}. Instead of the Lipschitz condition,
Srinivas et al.~\cite{SriKraKakSee09} exploit the structure of
Gaussian processes, focusing on the notion of the effective
dimension. These various ``non-parametric'' bandit problems typically
suffer from the curse of dimensionality, that is, the best possible
convergence rates (after $T$ queries) are typically of the form
$T^{\alpha}$, with the exponent $\alpha$ approaching $1$ for large
dimension $d$.

The question addressed in the present paper is: How can we leverage
\emph{convexity} of the mean cost function as a structural assumption?
The main contribution of the paper is an algorithm which achieves,
with high probability, an $\tilde{O}(\text{poly}(d)\sqrt{T})$ regret
after $T$ requests. This result holds for all convex Lipschitz mean
cost functions. We remark that the rate does not deteriorate with $d$
(except in the multiplicative term) implying that convexity is a
strong structural assumption which turns ``non-parametric'' Lipschitz
problems into ``parametric''. Nevertheless, convexity is a very
natural and basic assumption, and applications of our method are,
therefore, abundant. Let us also remark that $\Omega(\sqrt{dT})$ lower
bounds have been shown for linear mean cost
functions~\cite{DanHayKak08}, making our algorithm optimal up to
factors polynomial in the dimension $d$ and logarithmic in the number
of iterations $T$.

We note that our work focuses on the so-called \emph{stochastic}
bandits setting, where the observed costs of an action are
i.i.d.~draws from a fixed distribution. A parallel line of literature
focuses on the more difficult adversarial setting where the costs of
actions change arbitrarily from round to round. Leveraging structure
in non-stochastic bandit settings is more complex, and is not a goal
of this paper.

We start by defining some notation and the problem setup below. The
next section will survey related prior works and describe their
connections with our work in
Section~\ref{sec:outline}. Section~\ref{sec:id} gives the algorithm
and analysis for the special case of univariate optimization. The
algorithm for higher dimensions and its analysis are given in
Section~\ref{sec:highd}.

\paragraph{Notation and setup:} Let $\xset$ be a compact and convex
subset of $\Reals^d$, and let $f\colon\xset\to\Reals$ be a
$1$-Lipschitz convex function on $\xset$, so $f(x) - f(x') \leq
\|x-x'\|$ for all $x, x' \in \xset$.  We assume that $\xset$ is
specified in a way so that an algorithm can efficiently construct the
smallest Euclidean ball containing the set.  Furthermore, we assume
the algorithm has noisy black-box access to $f$.  Specifically, the
algorithm is allowed to query the value of $f$ at any $x \in \xset$,
and the response to the query $x$ is
\[ y = f(x) + \noise \]
where $\noise$ is an independent $\sigma$-subgaussian random variable
with mean zero: $\E[\exp(\lambda\noise)] \leq
\exp(\lambda^2\sigma^2/2)$ for all $\lambda \in \Reals$. The algorithm
incurs a cost $f(x)$ for each query $x$.  The goal of the algorithm is
to minimize its \emph{regret}: after making $T$ queries
$x_1,\dotsc,x_T \in \xset$, the regret of the algorithm is
\[ R_T = \sum_{t=1}^T \bigl( f(x_t) - f(x^*) \bigr) \]
where $x^*$ is the minimizer of $f$ over $\xset$ (we do not require
uniqueness of $x^*$). 

Since we observe noisy function values, our algorithms will make
multiple queries of $f$ at the same point. We will construct an
average and confidence interval (henceforth CI) around the average for
the function values at points queried by the algorithm. We will use
the notation $\LB_{\gamma_i}(x)$ and $\UB_{\gamma_i}(x)$ to denote the
lower and upper bounds of a CI of width $\gamma_i$ for the function
estimate of a point $x$. We will say that CI's at two points are
$\gamma$-separated if $\LB_{\gamma_i}(x) \geq \UB_{\gamma_i}(y) +
\gamma$ or $\LB_{\gamma_i}(y) \geq \UB_{\gamma_i}(x) + \gamma$.

\section{Related work}
\label{sec:related}

\emph{Asymptotic} rates of $\order(\sqrt{T})$ have been previously
achieved by Cope~\cite{Cope09} for unimodal functions under stringent
conditions (smoothness and strong convexity of the mean cost function,
in addition to the unconstrained optimum being achieved inside the
constraint set). The method employed by the author is a variant of the
classical Kiefer-Wolfowitz procedure \cite{KieWol52} for estimation of
an optimum point. Further, the rate $\tilde{O}(\sqrt{T})$ has been
achieved in Auer et al.~\cite{AueOrtSze07} for a one-dimensional
non-convex problem with finite number of optima. The result assumes
continuous second derivatives of the mean function, not vanishing at
the optimum, while the first derivative is assumed to be zero at the
optima. The method is based on discretizing the interval and does not
exploit convexity. Yu and Mannor~\cite{YuMa2011} recently studied
unimodal bandits, but they only consider one-dimensional and
graph-structured settings. Bubeck et al.~\cite{BubeckMuStSz2011}
consider the general setup of $\mathcal{X}$-armed bandits with
Lipschitz mean cost functions and their algorithm does give
$\order(c(d)\sqrt{T})$ regret for a dimension dependent constant
$c(d)$ in some cases when the problem has a near-optimality dimension
of 0. However, not all convex, Lipschitz functions satisfy this
condition, and $c(d)$ can grow exponentially in $d$ even in these
special cases.

The case of convex, Lipschitz cost functions has been looked at in the
harder adversarial model~\cite{flaxman05bgd,Kle05} by constructing
one-point gradient estimators. However, the best-known regret bounds
for these algorithms are $\order(T^{3/4})$. Agarwal et
al.~\cite{AgarwalDeXi2010} show a regret bound of $\order(\sqrt{T})$
in the adversarial setup, when two evaluations of the same function
are allowed, instead of just one. However, this does not include the
stochastic bandit optimization setting since each function evaluation
in the stochastic case is corrupted with independent noise, violating
the critical requirement of a bounded gradient estimator that their
algorithm exploits. Indeed, applying their result in our setup yields
a regret bound of $\order(T^{3/4})$.

A related line of work attempts to solve convex optimization problems
by instead posing the problem of finding a feasible point from a
convex set. Different oracle models of specifying the convex set
correspond to different optimization settings. The bandit setting is
identical to finding a feasible point, given only a membership oracle
for the convex set. Since we get only noisy function evaluations, we
in fact only have access to a noisy membership oracle. While there are
elegant solutions based on random walks in the easier separation
oracle model~\cite{Bertsimas2004}, the membership oracle setting has
been mostly studied in the noiseless setting only and uses much more
complex techniques building on the seminal work of Nemirovski and
Yudin~\cite{NemirovskiYu83}. The techniques have the
additional drawback that they do not guarantee a low regret since the
methods often explore aggressively.

We observe that the problem addressed in this paper is closely related
to noisy zero-th order (also called derivative-free) convex
optimization, whereby the algorithm queries a point of the domain and
receives a noisy value of the function. Given $\epsilon > 0$, such
algorithms are guaranteed to produce an $\epsilon$-minimizer at the
end of $T$ iterations. While the literature on stochastic optimization
is vast, we emphasize that an optimization guarantee does not
necessarily imply a bound on regret. We explain this point in more
detail below.

Since $f$ is convex by assumption, the average $\bar{x}_T =
\frac{1}{T}\sum_{t=1}^T x_t$ must satisfy $f(\bar{x}_T)-f(x^*) \leq
R_T/T$ (by Jensen's inequality). That is, a method guaranteeing small
regret is also an optimization algorithm. The converse, however, is
not necessarily true. Suppose an optimization algorithm queries $T$
points of the domain and then outputs a candidate minimizer
$x^*_T$. Without any assumption on the behavior of the optimization
method nothing can be said about the regret it suffers over $T$
iterations. In fact, depending on the particular setup, an
optimization method might prefer to spend time querying far from the
minimum of the function (that is, \emph{explore}) and then output the
solution at the last step. Guaranteeing a small regret typically
involves a more careful balancing of \emph{exploration} and
\emph{exploitation}. This distinction between arbitrary optimization
schemes and {\em anytime} methods is discussed further in the
paper~\cite{RagRak11}.

We note that most of the existing approaches to derivative-free
optimization outlined in the recent book~\cite{ConnScVi2009} typically
search for a descent or sufficient descent direction and then take a
step in this direction. However, most convergence results are
asymptotic and do not provide concrete rates even in an optimization
error setting. The main emphasis is often on global optimization of
non-convex functions, while we are mainly interested in convex
functions in this work. Nesterov~\cite{Nesterov2011} recently analyzes
schemes similar to that of Agarwal et al.~\cite{AgarwalDeXi2010} with
access to \emph{noiseless} function evaluations, showing
$\order(\sqrt{dT})$ convergence for non-smooth functions and
accelerated schemes for smooth mean cost functions. However, when
analyzed in a noisy evaluation setting, his rates suffer from the
degradation as those of Agarwal et al.~\cite{AgarwalDeXi2010}.

\section{Outline of our approach}
\label{sec:outline}

The close relationship between convex optimization and the
regret-minimization problem suggests a plan of attack: Check whether
existing stochastic zeroth order optimization methods (that is,
methods that only query the oracle for function values), in fact,
minimize regret. Two types of methods for stochastic zeroth order
convex optimization are outlined in Nemirovski and Yudin~\cite[Chapter
  9]{NemirovskiYu83}. The first approach uses the noisy function
values to estimate a gradient direction at every step, and then passes
this information to a stochastic first-order method. The second
approach is to use the zeroth order information to estimate function
values and pass this information to a \emph{noiseless} zeroth order
method. Nemirovski and Yudin argue that the latter approach has greater stability when
compared to the former. Indeed, for a gradient estimate to be
meaningful, function values should be sampled close to the point of
interest, which, in turn, results in a poor quality of the
estimate. This tension is also the source of difficulty in minimizing
regret with a convex mean cost function. 

Owing to the insights of Nemirovski and Yudin~\cite{NemirovskiYu83}, we opt for the
second approach, giving up the idea of estimating the first-order
information. The main novel tool of the paper is a ``center-point
device'' that allows to quickly detect that the optimization method
might be paying high regret and to act on this information. Unlike
discretization-based methods, the proposed algorithm uses convexity in
a crucial way. We first demonstrate the device on one-dimensional
problems, where the solution is clean and intuitive. We then develop a
version of the algorithm for higher dimensions, basing our
construction on the beautiful zero-th order optimization method of
Nemirovski and Yudin~\cite{NemirovskiYu83}. Their method does not
guarantee vanishing regret by itself, and a careful fusion of this
algorithm with our center-point device is required. The overall
approach would be to use center-point device in conjunction with a
modification of the classical ellipsoid algorithm.

To motivate the center-point device, consider the following
situation. Suppose $f$ is the unknown function on $\xset=[0,1]$, and
assume for now that it is linear with a slope $T^{-1/3}$. Let us
sample function values at $x=1/4$ and $x=3/4$. To even distinguish the
slope from a slope $-T^{-1/3}$ (which results in a minimizer on the
opposite side of $\xset$), we need $O(T^{2/3})$ points. If the
function $f$ is linear indeed, we only incur $O(T^{1/3})$ regret on
these rounds. However, if instead $f$ is a quadratic dipping between
the sampled points, we incur regret of $O(T^{2/3})$. To quickly detect
that the function is not flat between the two sampled points, we
additionally sample at $x=1/2$. The center point acts as a
\emph{sentinel}: if it is recognized that the function value at the
center point is noticeably below the other two values, the region
$[0,1/4]\cup [3/4,1]$ can be discarded. If it is recognized that the
value of $f$ either at $x=1/4$ or at $x=3/4$ is greater than others,
then either $[0,1/4]$ or $[3/4,1]$ can be discarded. Finally, if $f$
at all three points appears to be similar at a given scale, we have a
certificate that the algorithm is not paying regret larger than this
scale per query. The remaining argument proceeds similarly to the
binary search or the method of centers of gravity: since a constant
portion of the set is discarded every time, it only requires a
logarithmic number of ``cuts''. We remark that this novelty is indeed
in ensuring that regret is kept small in the process; a simpler
algorithm which does not query the center is sufficient to guarantee a
small optimization error but incurs a large regret on examples of the
form sketched above. 

In the next section we present the algorithm that results from the
above ideas for one-dimensional convex optimization. The general case
in higher dimensions is presented in Section~\ref{sec:highd}.

\section{One-dimensional case}
\label{sec:id}

We start with a specialization of the setting to 1-dimension to
illustrate some of the key ideas including the center-point device. We
assume without loss of generality that the domain $\xset = [0,1]$, and
$f(x) \in [0,1]$ (the latter can be achieved by pinning $f(x^*) = 0$
since $f$ is 1-Lipschitz). 

\subsection{Algorithm description}

\begin{algorithm}
\caption{One-dimensional stochastic convex bandit algorithm}
\label{alg:1d}
\begin{algorithmic}[1]
\INPUT noisy black-box access to $f\colon[0,1]\to\Reals$, total
number of queries allowed $T$.

\STATE Let $l_1 := 0$ and $r_1 := 1$.

\FOR{epoch $\epcount = 1, 2, \dotsc$}

  \STATE Let $w_\epcount := r_\epcount - l_\epcount$.

  \STATE Let $x_l := l_\epcount + w_\epcount/4$, $x_c := l_\epcount +
  w_\epcount/2$, and $x_r := l_\epcount + 3w_\epcount/4$.

  \FOR{round $i = 1, 2, \dotsc$}
    \STATE Let $\gamma_i := 2^{-i}$.

    \STATE For each $x \in \{ x_l, x_c, x_r \}$, query $f(x)$
    $\frac{2\sigma}{\gamma_i^2}\log T$ times.

    \IF{$\max \{ \LB_{\gamma_i}(x_l),
      \LB_{\gamma_i}(x_r)\} \geq \min \{ \UB_{\gamma_i}(x_l),
      \UB_{\gamma_i}(x_r) \} + \gamma_i$}
       
    \STATE\qquad\qquad\qquad\qquad\qquad\qquad\qquad\qquad\qquad\COMMENT{\textbf{Case 1}: \texttt{CI's at $x_l$ and $x_r$ are $\gamma_i$
        separated}}
    \SHORTIF{$\LB_{\gamma_i}(x_l) \geq \LB_{\gamma_i}(x_r)$} 
    let $l_{\epcount+1} := x_l$ and $r_{\epcount+1} := r_\epcount$.

    \SHORTIF{$\LB_{\gamma_i}(x_l) < \LB_{\gamma_i}(x_r)$}
    let $l_{\epcount+1} := l_\epcount$ and $r_{\epcount+1} := x_r$.
    
    \STATE Continue to epoch $\epcount+1$.
    
    \ELSIF{$\max\{ \LB_{\gamma_i}(x_l),
      \LB_{\gamma_i}(x_r) \} \geq \UB_{\gamma_i}(x_c) + \gamma_i$}

    \STATE\qquad\qquad\qquad\qquad\qquad\qquad\qquad\qquad\COMMENT{\textbf{Case 2}: \texttt{CI's at $x_c$ and $x_l$ or $x_r$ are
        $\gamma_i$ separated}}
      \SHORTIF{$\LB_{\gamma_i}(x_l) \geq \LB_{\gamma_i}(x_r)$}
      let $l_{\epcount+1} := x_l$ and $r_{\epcount+1} := r_\epcount$.

      \SHORTIF{$\LB_{\gamma_i}(x_l) < \LB_{\gamma_i}(x_r)$}
      let $l_{\epcount+1} := l_\epcount$ and $r_{\epcount+1} := x_r$.

      \STATE Continue to epoch $\epcount+1$.

    \ENDIF
  \ENDFOR
\ENDFOR

\end{algorithmic}
\end{algorithm}

Algorithm~\ref{alg:1d} proceeds in a series of \emph{epochs}
demarcated by a working feasible region (the interval
$[l_\epcount,r_\epcount]$ in epoch $\epcount$).  In each epoch, the
algorithm aims to discard a portion of the working feasible region
determined to only contain suboptimal points.  To do this, the
algorithm repeatedly makes noisy queries to $f$ at three different
points in the working feasible region. Each epoch is further
subdivided into \emph{rounds}, where we query the function
$(2\sigma\log T)/\gamma_i^2$ times in round $i$ at each of the
points. By Hoeffding's inequality, this implies that we know the
function value to within $\gamma_i$ with high probability. The value
$\gamma_i$ is halved at every round so that the algorithm can stop the
epoch with the minimal number of queries that suffice to resolve the
difference between function values at any two of $x_l, x_c, x_r$,
ensuring a low regret regret in each epoch.  At the end of an epoch
$\epcount$, the working feasible region is reduced to a subset
$[l_{\epcount+1},r_{\epcount+1}] \subset [l_\epcount,r_\epcount]$ of
the current region for the next epoch $\epcount+1$, and this reduction
is such that the new region is smaller in size by a constant fraction.
This geometric rate of reduction guarantees that only a small number
of epochs can occur before the working feasible region only contains
near-optimal points.

In order for the algorithm to identify a sizable portion of the
working feasible region containing only suboptimal points to discard,
the queries in each epoch should be suitably chosen, and the convexity
of $f$ must be judiciously exploited.  To this end, the algorithm
makes its queries at three equally-spaced points $x_l < x_c < x_r$ in
the working feasible region.
\begin{description}
\item[\textbf{Case 1:}] If the confidence intervals around $f(x_l)$
  and $f(x_r)$ are sufficiently separated, then the algorithm can
  identify a subset of the feasible region (either to the left of
  $x_l$ or to the right of $x_r$) that contains no near-optimal
  points---\emph{i.e.}, that every point $x$ in the subset has $f(x)
  \gg f(x^*)$.  This subset, which is a fourth of the working feasible
  region by construction is then discarded and the algorithm continues
  to the next epoch. This case is depicted in
  Figure~\ref{fig:1dcases-1a}. 
  \begin{figure}[htbp]
    \centering
    \includegraphics[height=1.5in]{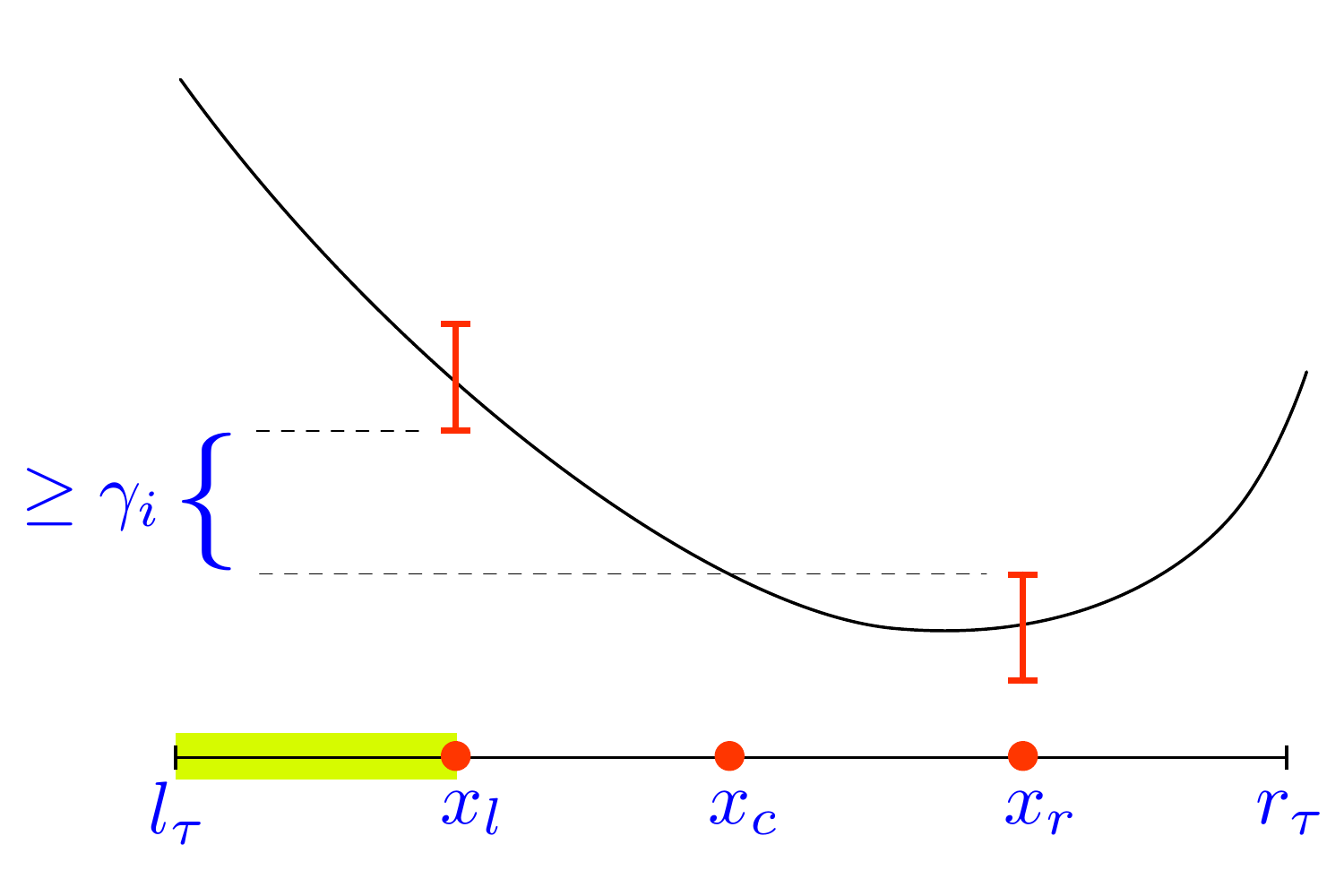}
    \hspace{2cm}
    \includegraphics[height=1.5in]{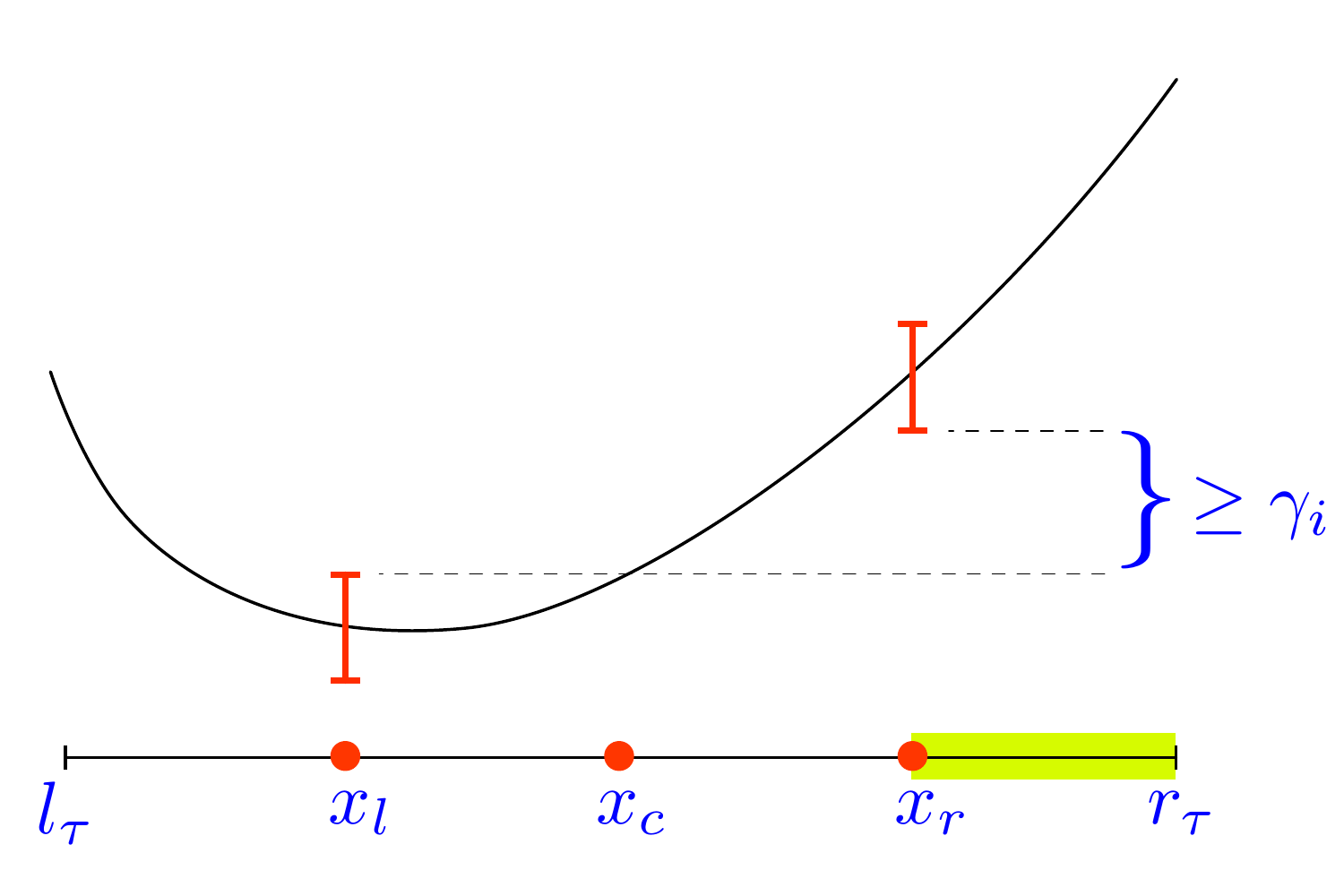}
    \caption{Two possible configurations when the algorithm enters case 1.}
    \label{fig:1dcases-1a}
  \end{figure}

\item[\textbf{Case 2:}] If the above deduction cannot be made, the
  algorithm looks at the confidence interval around $f(x_c)$.  If this
  interval is sufficiently below at least one of the other intervals
  (for $f(x_l)$ or $f(x_r)$), then again the algorithm can identify a
  quartile that contains no near-optimal points, and this quartile can
  then be discarded before continuing to the next epoch. One possible
  arrangement of CI's for this case is shown in
  Figure~\ref{fig:1dcases-2a}.

\item[\textbf{Case 3:}] Finally, if none of the earlier cases is true,
  then the algorithm is assured that the function is sufficiently flat
  on working feasible region and hence it has not incurred much regret
  so far. The algorithm continues the epoch, with an increased number
  of queries to obtain smaller confidence intervals at each of the
  three points. An example arrangement of CI's for this case is shown
  in Figure~\ref{fig:flat1d}.
	
  \begin{figure}[ht]
    \begin{minipage}[b]{0.47\linewidth}
      \centering
      \includegraphics[height=1.5in]{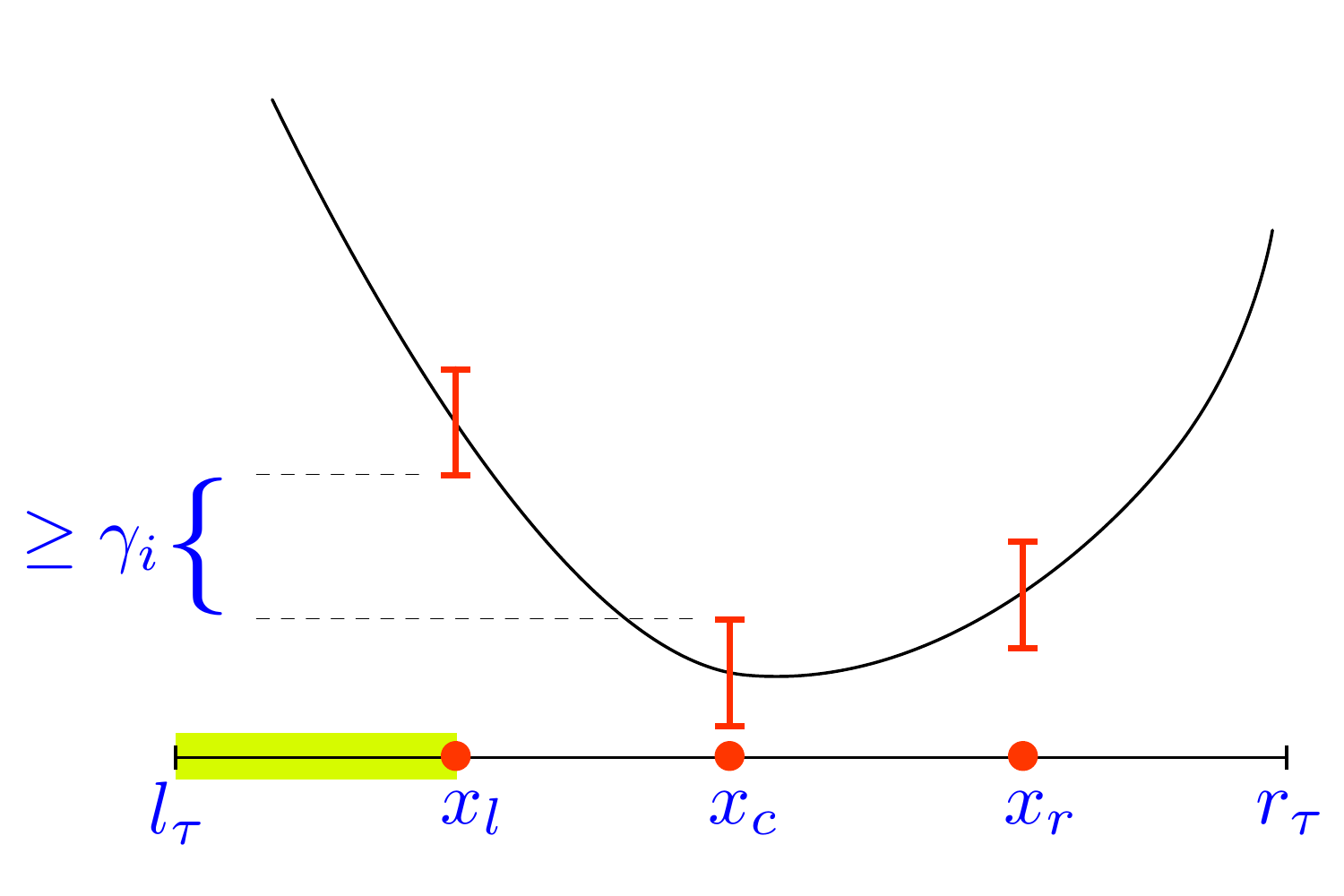}
      \caption{One of the possible configurations when
        the algorithm enters Case 2.}
      \label{fig:1dcases-2a}  
    \end{minipage}
    \hspace{0.5cm}
    \begin{minipage}[b]{0.47\linewidth}
      \centering
      \includegraphics[height=1.5in]{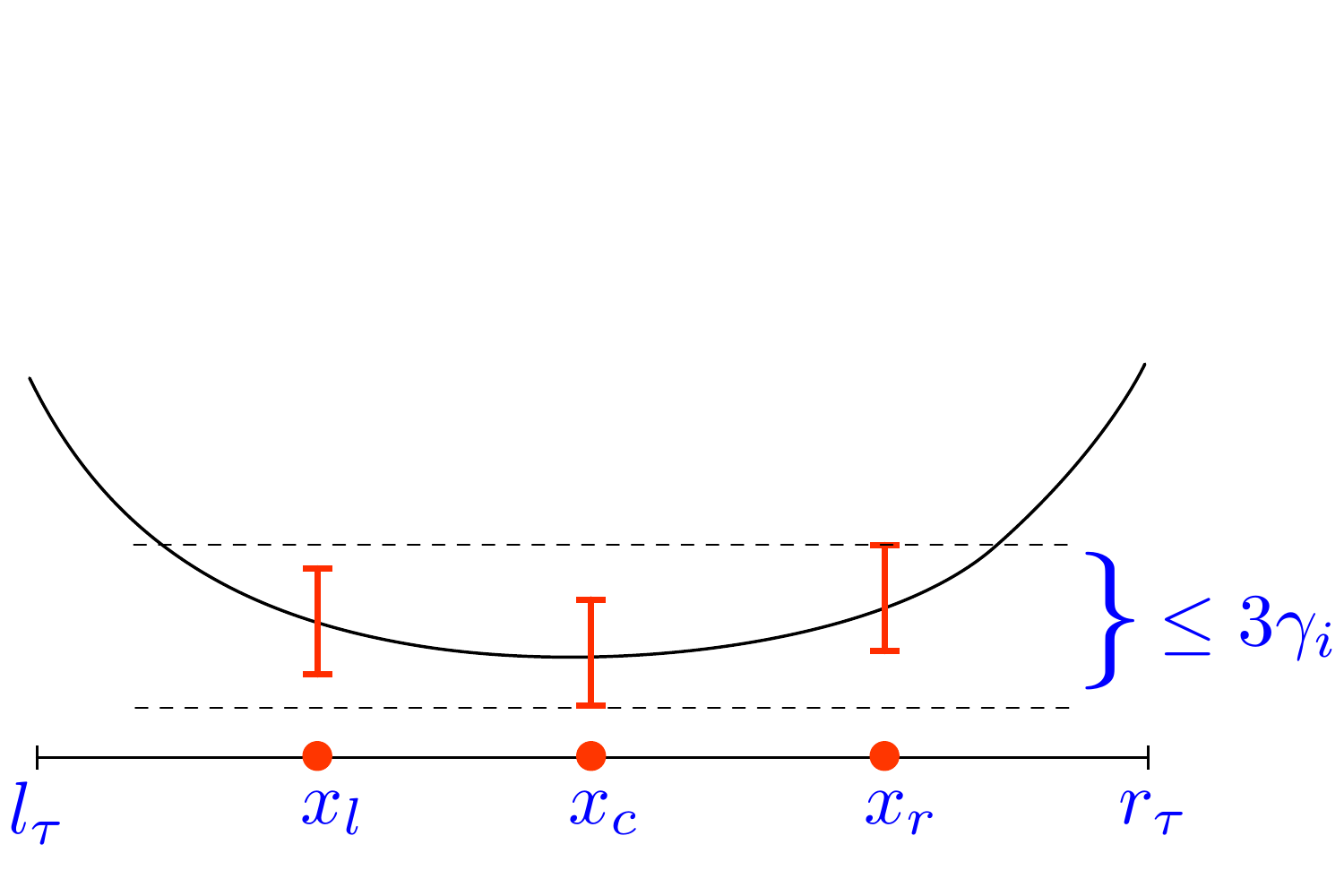}
      \caption{Configuration of the confidence intervals
        in Case 3 of Algorithm~\ref{alg:1d}.}
      \label{fig:flat1d}
    \end{minipage}
  \end{figure}

\end{description}

\subsection{Analysis}

The analysis of Algorithm~\ref{alg:1d} relies on the function values
being contained in the confidence intervals we construct at each round
of each epoch. To avoid having probabilities throughout our analysis,
we define an event $\event$ where at each epoch $\epcount$, and each
round $i$, $f(x) \in [\LB_{\gamma_i}(x),\UB_{\gamma_i}(x)]$ for $x \in
\{x_l, x_c, x_r\}$. We will carry out the remainder of the analysis
conditioned on $\event$ and bound the probability of $\event^c$ at the
end.

The following theorem bounds the regret incurred by
Algorithm~\ref{alg:1d}. We note that the regret would be maintained in
terms of the points $x_t$ queried by the algorithm at time $t$. Within
any given round, the order of queries is immaterial to the regret.

\begin{theorem}[Regret bound for Algorithm~\ref{alg:1d}] 
\label{theorem:regret-1d}
Suppose Algorithm~\ref{alg:1d} is run on a convex, 1-Lipschitz
function $f$ bounded in [0,1]. Suppose the noise in observations is
i.i.d. and $\sigma$-subgaussian. Then with probability at least $1 -
1/T$ we have

\begin{equation*}
\sum_{t=1}^T f(x_t) - f(x^*) \leq 108\sqrt{\sigma T\log T}
\log_{4/3}\left(\frac{T}{8\sigma\log T}\right).  
\end{equation*}
\end{theorem}

\paragraph{Remarks:} As stated Algorithm~\ref{alg:1d} and
Theorem~\ref{theorem:regret-1d} assume knowledge of $T$, but we can
make the algorithm adaptive to $T$ by a standard doubling argument. We
remark that $\order(\sqrt{T})$ is the smallest possible regret for any
algorithm even with noisy gradient information. Hence, this result
shows that for purposes of regret, noisy zeroth order information is
no worse than noisy first-order information apart from logarithmic
factors. We also observe that at the end of the procedure, the
mid-point $x_c$ of the working feasible region $[l_{\epcount},
  r_{\epcount}]$ where $\epcount$ was the last epoch, has an
optimization error of at most $\tilde{\order}(1/\sqrt{T})$. This is
unlike noisy first-order methods where all the iterates have to be
averaged in order to get a point with low optimization error.

The theorem is proved via a series of lemmas in the next few
sections. The key idea is to show that the regret on any epoch is
small and the total number of epochs is bounded. To bound the
per-epoch regret, we will show that the total number of queries made
on any epoch depends on how close to flat the function is on the
working feasible region. Thus we either take a long time, but the
function is very flat, or we stop early when the function has
sufficient slope, never accruing too much regret.
\subsubsection{Bounding the regret in one epoch}
 
We start by showing that each reduction in the working feasible region
after each epoch never discards near-optimal points.

\begin{lemma}
\label{lemma:survival-1d}
If epoch $\epcount$ ends in round $i$, then the interval
$[l_{\epcount+1},r_{\epcount+1}]$ contains every $x \in
[l_\epcount,r_\epcount]$ such that $f(x) \leq f(x^*) + \gamma_i$.
In particular, $x^* \in [l_\epcount,r_\epcount]$ for all epochs $\epcount$.
\end{lemma}
\begin{proof}
Suppose epoch $\epcount$ terminates in round $i$ via case 1.
This means that either
$\LB_{\gamma_i}(x_l) \geq \UB_{\gamma_i}(x_r) + \gamma_i$
or
$\LB_{\gamma_i}(x_r) \geq \UB_{\gamma_i}(x_l) + \gamma_i$.
Consider the former case (the argument for the latter is analogous).
This implies

\begin{equation} 
  \label{eqn:1dgammasep}
  f(x_l) \geq f(x_r) + \gamma_i . 
\end{equation}
We need to show that every $x \in [l_\epcount,l_{\epcount+1}] =
[l_\epcount,x_l]$ has $f(x) \geq f(x^*) + {\gamma_i}$.
So pick $x \in [l_\epcount,x_l]$ so that $x_l \in [x,x_r]$.
Then $x_l = tx + (1-t)x_r$ for some $0 \leq t\leq 1$, so
by convexity,
\[ f(x_l) \leq t f(x) + (1-t) f(x_r) , \]
which in turn implies
\begin{align*}
f(x) & \geq f(x_r) + \frac{f(x_l) - f(x_r)}{t} \\ 
  & \geq f(x_r) + \frac{\gamma_i}{t} \qquad\mbox{using 
  Equation~\ref{eqn:1dgammasep}} \\  
  & \geq f(x^*) + \gamma_i \qquad\mbox{since $t \leq 1$}
\end{align*}
as required. 

Now suppose epoch $\epcount$ terminates in round $i$ via case 2.
This means
\[
\max\{ \LB_{\gamma_i}(x_l), \LB_{\gamma_i}(x_r) \}
\geq \UB_{\gamma_i}(x_c) + \gamma_i
.
\]
Suppose $\LB_{\gamma_i}(x_l) \geq \LB_{\gamma_i}(x_r)$ (the argument for the
case $\LB_{\gamma_i}(x_l) < \LB_{\gamma_i}(x_r)$ is analogous).
The above inequality implies
\[ f(x_l) \geq f(x_c) + \gamma_i . \]
We need to show that every $x \in [l_\epcount,l_{\epcount+1}] =
[l_\epcount,x_l]$ has $f(x) \geq f(x^*) + \gamma_i$.
But the same argument as given in case 1, with $x_r$ replaced with $x_c$,
gives the required claim.

The fact that $x^* \in [l_\epcount,r_\epcount]$ for all epochs $\epcount$
follows by induction.
\end{proof}

The next two lemmas bound the regret incurred in any single epoch.  To
show this, we first establish that an algorithm incurs low regret in a
round as long as it does not end an epoch.  Then, as a consequence of
the doubling trick, we show that the regret incurred in an epoch is on
the same order as that incurred in the last round of the epoch.

\begin{lemma}[Certificate of low regret]
\label{lemma:certificate-1d}
If epoch $\epcount$ continues from round $i$ to round $i+1$, then the
regret incurred in round $i$ is at most
\[ \frac{72\sigma\log T}{\gamma_i} . \]
\end{lemma}
\begin{remark}
A more detailed argument shows that the regret incurred is, in fact, at
most $54\sigma\log T/\gamma_i$.
\end{remark}
\begin{proof}
The regret incurred in round $i$ of epoch $\epcount$ is
\[ \frac{2\sigma\log T}{\gamma_i^2}
\cdot \Bigl(
(f(x_l) - f(x^*))
+ (f(x_c) - f(x^*))
+ (f(x_r) - f(x^*))
\Bigr)
\]
so it suffices to show that
\[ f(x) \leq f(x^*) + 12\gamma_i \]
for each $x \in \{ x_l, x_c, x_r \}$.

The algorithm continues from round $i$ to round $i+1$ iff
\[ \max\{\LB_{\gamma_i}(x_l), \LB_{\gamma_i}(x_r) \} < \min\{
\UB_{\gamma_i}(x_l), \UB_{\gamma_i}(x_r) \} + \gamma_i \]
and
\[
\max\{ \LB_{\gamma_i}(x_l), \LB_{\gamma_i}(x_r) \}
< \UB_{\gamma_i}(x_c) + \gamma_i
.
\]
This implies that $f(x_l)$, $f(x_c)$, and $f(x_r)$ are contained in an
interval of width at most $3\gamma_i$ (recall
Figure~\ref{fig:flat1d}). 

By Lemma~\ref{lemma:survival-1d}, we have $x^* \in
[l_\epcount,r_\epcount]$.
Assume $x^* \leq x_c$ (the case $x^* > x_c$ is analogous).
There exists $t \geq 0$ such that $x^* = x_c + t(x_c - x_r)$, so
\[ x_c = \frac1{1+t} x^* + \frac{t}{1+t} x_r. \]
Note that $t \leq 2$ because $|x_c - l_\epcount| = w_\epcount/2$ and $|x_r
- x_c| = w_\epcount/4$, so
\[
t = \frac{|x^*-x_c|}{|x_r-x_c|}
\leq \frac{|l_\epcount-x_c|}{|x_r-x_c|}
= \frac{w_\epcount/2}{w_\epcount/4} = 2
.
\]
By convexity,
\[ f(x_c) \leq \frac1{1+t} f(x^*) + \frac{t}{1+t} f(x_r) \]
so
\begin{align*}
f(x^*) & \geq (1+t) \left( f(x_c) - \frac{t}{1+t} f(x_r) \right) \\
& = f(x_r) + (1+t) \left( f(x_c) - f(x_r) \right) \\
& \geq f(x_r) - (1+t) |f(x_c) - f(x_r)| \\
& \geq f(x_r) - (1+t) \cdot 3\gamma_i \\
& \geq f(x_r) - 9 \gamma_i
.
\end{align*}
We conclude that for each $x \in \{ x_l, x_c, x_r \}$,
\[
f(x)
\leq f(x_r) + 3\gamma_i 
\leq f(x^*) + 12\gamma_i
.
\qedhere
\]
\end{proof}

\begin{lemma}[Regret in an epoch] \label{lemma:epoch-regret-1d}
If epoch $\epcount$ ends in round $i$, then the regret incurred in the
entire epoch is
\[ \frac{216\sigma\log T}{\gamma_i} . \]
\end{lemma}
\begin{proof}
If $i = 1$, then $f(x) - f(x^*) \leq |x-x^*| \leq 1$ for each $x \in
\{x_l,x_c,x_r\}$ because $f$ is $1$-Lipschitz and $|x-x'| \leq 1$ for any
$x,x' \in [0,1]$.
Therefore, the regret incurred in epoch $\epcount$ is
\[ \frac{2\sigma\log T}{\gamma_1^2}
\cdot \Bigl(
(f(x_l) - f(x^*))
+ (f(x_c) - f(x^*))
+ (f(x_r) - f(x^*))
\Bigr)
\leq \frac{12\sigma\log T}{\gamma_1}
.
\]
Now assume $i \geq 2$.
Lemma~\ref{lemma:certificate-1d} implies that the regret incurred in round
$j$, for $1 \leq j \leq i-1$, is at most
\[
\frac{72\sigma\log T}{\gamma_j}.
\]
Furthermore, for round $i$, we still know that the regret on each
query in round $i$ is bounded by $36\gamma_{i-1}$ ($12\gamma_{i-1}$
for each of $x_l$, $x_c$, $x_r$). Recalling that $\gamma_{i-1} =
2\gamma_i$ and that we make $(\sigma\log T)/\gamma_i^2$ queries at
round $i$, the regret incurred in round $i$ (the final round of epoch
$\epcount$) is at most
\[ 36\gamma_{i-1}\frac{2\sigma\log T}{\gamma_i^2} =
\frac{144\sigma\log T}{\gamma_i}.
 \]
Therefore, the overall regret incurred in epoch $\epcount$ is
\[
\sum_{j=1}^{i-1} \frac{72\sigma\log T}{\gamma_j} + \frac{144\sigma\log
  T}{\gamma_i} = \sum_{j=1}^{i-1} 72\sigma\log T \cdot 2^j +
\frac{144\sigma\log T}{\gamma_i} < 72\sigma\log T \cdot 2^i +
\frac{144\sigma\log T}{\gamma_i} = \frac{216\sigma\log T}{\gamma_i} .
\qedhere
\]
\end{proof}

\subsubsection{Bounding the number of epochs}

To establish the final bound on the overall regret, we bound the
number of epochs that can occur before the working feasible region
only contains near-optimal points.  The final regret bound is simply
the product of the number of epochs and the regret incurred in any
single epoch.

\begin{lemma}[Bound on the number of epochs] \label{lemma:epoch-bound-1d}
The total number of epochs $\epcount$ performed by
Algorithm~\ref{alg:1d} is at bounded as 
\[
\epcount \leq \frac{1}{2} \log_{4/3}\left(\frac{ T}{8\sigma\log
  T}\right).
\]
\end{lemma}

\begin{proof}

The proof is based on observing that $\gamma_i \geq (T/2\sigma\log
T)^{-1/2}$ at all epochs and rounds. Indeed if $\gamma_i \leq
(T/2\sigma\log T)^{-1/2}$, step 7 of the algorithm would require more
than $T$ queries to get the desired confidence intervals in that
round. Hence we set $\gamma_{\min} = (T/2\sigma\log T)^{-1/2}$ and
define the interval $I := [x^* - \gamma_{\min}, x^* + \gamma_{\min}]$
which has width $2\gamma_{\min}$.  For any $x \in I$,

\[ f(x) - f(x^*) \leq |x - x^*| \leq \gamma_{\min} \]
because $f$ is $1$-Lipschitz.
Moreover, for any epoch $\epcount'$ which ends in round $i'$,
$\gamma_{\min} \leq \gamma_{i'} $ by definition  and therefore by
Lemma~\ref{lemma:survival-1d}, 
\[ I \subseteq \{ x \in [0,1] \colon f(x) \leq f(x^*) + \gamma_{i'} \} 
\subseteq [l_{\epcount'+1},r_{\epcount'+1}]
.
\]
This implies that
\[ 2\gamma_{\min} \leq r_{\epcount+1} - l_{\epcount+1} =
w_{\epcount+1} . \]
Furthermore, by the definitions of $l_{\epcount'+1}$, $r_{\epcount'+1}$,
and $w_{\epcount'+1}$ in the algorithm, it follows that
\[ w_{\epcount'+1} \leq \frac34 \cdot w_{\epcount'} \]
for any $\epcount' \in \{1,\dotsc,\epcount\}$.
Therefore, we conclude that
\[ 2\gamma_{\min} \leq w_{\epcount+1} \leq
\left(\frac34\right)^\epcount \cdot w_1 = \left(\frac34\right)^\epcount
\]
which gives the claim after rearranging the inequality.
\end{proof}

\subsubsection{Proof of Theorem~\ref{theorem:regret-1d}}

The statement of the theorem follows by combining the per-epoch regret
bound of Lemma~\ref{lemma:epoch-regret-1d} with the above bound on the
number of epochs, and showing that all these bounds hold with
sufficiently high probability. 

Lemma~\ref{lemma:epoch-regret-1d} implies that the regret incurred in
any epoch $\epcount' \leq \epcount$ that ends in round $i'$ is at most

\[ \frac{216\sigma\log T}{\gamma_{i'}} \leq \frac{216\sigma\log
  T}{\gamma_{\min}} \leq 216\sqrt{T\sigma\log T} . \] 
So the overall regret incurred in all $\epcount$
epochs is at most
\[ 216\sqrt{T\sigma\log T} \cdot
\frac12\log_{4/3}\left(\frac{T}{8\sigma\log T}\right)  
.
\qedhere
\]

Finally we recall that the entire analysis thus far has been
conditioned on the event $\event$ where all the confidence intervals
we construct do contain the function values. We would now like to
control the probability $\P(\event^c)$. Consider a fixed round and a
fixed point $x$. Then after making $2\sigma\log T/\gamma_i^2$ queries,
Hoeffding's inequality gives that

\begin{align*}
  \P\left(|f(x) - \hat{f}(x)| \geq \gamma_i\right) \leq
  \frac{1}{T^2},
\end{align*}
where $\hat{f}(x)$ is the average of the observed function
values. Once we have a bound for a fixed round of a fixed epoch, we
would like to bound this probability uniformly over all rounds played
across all epochs. We note that we make at most $T$ queries, which is
also an upper bound on the total number of rounds. Hence union bound
gives 
\begin{equation*}
  \P(\event^c) \leq \frac{1}{T},
\end{equation*}
which completes the proof of the theorem. $\qed$

\section{Algorithm for optimization in higher dimensions}
\label{sec:highd}

We now move to present the general algorithm that works in
$d$-dimensions. The natural approach would be to try and generalize
Algorithm~\ref{alg:1d} to work in multiple dimensions. However, the
obvious extension requires constructing a covering of the unit sphere
and querying the function along every direction in the covering so
that we know the behavior of the function along every direction. While
such an approach yields regret that scales as $\sqrt{T}$, the
dependence on dimension $d$ is exponential both in regret and the
running time. The same problem was encountered in the scenario of
zeroth order optimization by Nemirovski and Yudin~\cite{NemirovskiYu83}, and they use a
clever construction to capture all the directions in polynomially many
queries. We define a pyramid to be a $d$-dimensional polyhedron
defined by $d+1$ points; $d$ points form a $d$-dimensional regular
polygon that is the base of the pyramid, and the apex lies above the
hyperplane containing the base (see Figure~\ref{fig:pyramid-basic} for
a graphical illustration in 3 dimensions). The idea of Nemirovski and Yudin was to build
a sequence of pyramids, each capturing the variation of function in
certain directions, in such a way that in $\order(d\log d)$ pyramids
we can explore all the directions. However, as mentioned earlier,
their approach fails to give a low regret. We combine their geometric
construction with ideas from the one-dimensional case to obtain a
low-regret algorithm as described in Algorithm~\ref{alg:highd}
below. Concretely, we combine the geometrical construction of
Nemirovski and Yudin~\cite{NemirovskiYu83} with the center-point device to show low
regret.

\begin{figure}[thbp]
\centering
\includegraphics[height=1.5in]{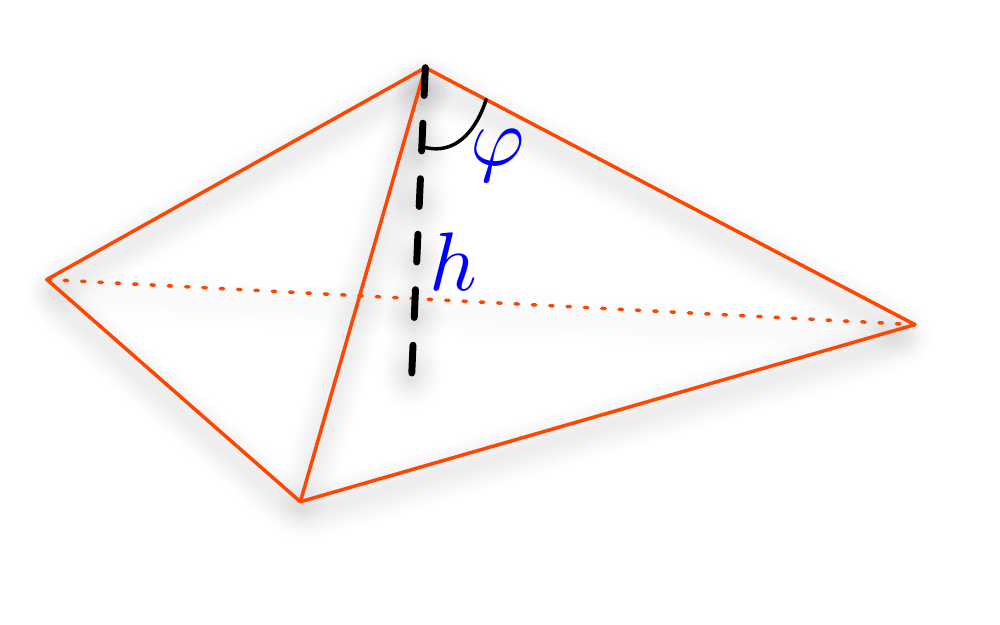}
\caption{Pyramid in 3-dimensions}
\label{fig:pyramid-basic}
\end{figure}

\begin{algorithm}
\caption{Stochastic convex bandit algorithm}
\begin{algorithmic}[1]
\INPUT feasible region $\xset \subset \Reals^d$; noisy black-box
access to $f\colon\xset\to\Reals$, constants $\inconst$ and
$\phiconst$, functions $\Delta_{\tau}(\gamma)$,
$\otherdelta_{\tau}(\gamma)$ and number of queries $T$ allowed.

\STATE Let $\xset_1 := \xset$.

\FOR{epoch $\epcount = 1, 2, \dotsc$} \STATE Round $\xset_\epcount$ so
$\ball(r_\epcount) \subseteq \xset_\epcount \subseteq
\ball(\outrad_\epcount)$, $\outrad_\epcount$ is minimized, and
$\inrad_\epcount := \outrad_\epcount /(\inconst d)$. Let
$\Ball_{\epcount} = \ball(\outrad_{\epcount})$.

  \STATE Construct regular simplex with vertices $x_1,\dotsc,x_{d+1}$ on
  the surface of $\ball(r_\epcount)$.

  \FOR{round $i = 1, 2, \dotsc$}
    \STATE Let $\gamma_i := 2^{-i}$.

    \STATE Query $f$ at $x_j$ for each $j = 1,\dots,d+1$
    $\frac{2\sigma\log T}{\gamma_i^2}$ times. 

    \STATE Let $\apex_1 := \arg\max_{x_j} \LB_{\gamma_i}(x_j)$.

    \FOR{pyramid $k = 1, 2, \dotsc$}
      \STATE Construct pyramid $\pyramid_k$ with apex $\apex_k$;
      let $z_1,\dotsc,z_d$ be the vertices of the base of $\pyramid_k$ and
      $z_0$ be the center of $\pyramid_k$.

      \STATE Let $\hatgamma := 2^{-1}$.

      \LOOP 
      \STATE Query $f$ at each of $\{y_k,z_0,z_1,\dotsc,z_d\}$
      $\frac{2\sigma\log T}{{\hatgamma}^2}$ times.

        \STATE Let $\pyracenter := z_0$, $\apexsc := y_k$, $\top$ be the
        vertex $v$ of $\pyramid_k$ maximizing $\LB_\hatgamma(v)$, $\bottom$
        be the vertex $v$ of $\pyramid_k$ minimizing $\LB_\hatgamma(v)$.
        \IF{$\LB_\hatgamma(\top) \geq \UB_\hatgamma(\bottom) +
        \Delta_\epcount(\hatgamma)$ and $\LB_\hatgamma(\top) \geq
        \UB_\hatgamma(\apexsc) + \hatgamma$}

          \STATE \COMMENT{Case 1(a)}

          \STATE Let $y_{k+1} := \top$, and immediately continue to pyramid
          $k+1$.

        \ELSIF{$\LB_\hatgamma(\top) \geq \UB_\hatgamma(\bottom) +
        \Delta_\epcount(\hatgamma)$ and $\LB_\hatgamma(\top) <
        \UB_\hatgamma(\apexsc) + \hatgamma$}
          \STATE \COMMENT{Case 1(b)}

          \STATE Set $(\xset_{\epcount+1},\Ball^{'}_{\epcount+1})$ =
          \textsc{Cone-cutting}$(\pyramid_k, \xset_{\epcount},
          \Ball_{\epcount})$, and proceed to epoch $\epcount+1$.

        \ELSIF{$\LB_\hatgamma(\top) < \UB_\hatgamma(\bottom) +
        \Delta_\epcount(\hatgamma)$ and $\UB_\hatgamma(\pyracenter) \geq
        \LB_\hatgamma(\bottom) - \otherdelta_{\epcount}(\hatgamma)$}

          \STATE \COMMENT{Case 2(a)}

          \STATE Let $\hatgamma := \hatgamma / 2$.

          \SHORTIF{$\hatgamma < \gamma_i$} start next round $i+1$.

        \ELSIF{$\LB_\hatgamma(\top) < \UB_\hatgamma(\bottom) +
        \Delta_\epcount(\hatgamma)$ and $\UB_\hatgamma(\pyracenter) <
        \LB_\hatgamma(\bottom) - \otherdelta_{\epcount}(\hatgamma)$}

          \STATE \COMMENT{Case 2(b)}

          \STATE Set $(\xset_{\epcount+1}, \Ball^{'}_{\epcount+1})$=
          \textsc{Hat-raising}$(\pyramid_k, \xset_{\epcount},
          \Ball_{\epcount})$, and proceed to epoch $\epcount+1$.

        \ENDIF
      \ENDLOOP
    \ENDFOR
  \ENDFOR
\ENDFOR

\end{algorithmic}
\label{alg:highd}
\end{algorithm}

\begin{algorithm}
\caption{\textsc{Cone-cutting}}
\begin{algorithmic}[1]
\INPUT pyramid $\pyramid$ with apex $\apex$, (rounded) feasible region
$\xset_\epcount$ for epoch $\epcount$, enclosing ball
$\Ball_{\epcount}$
\STATE Let $\base_1,\dots,\base_d$ be the vertices of the base of
$\pyramid$, and $\wideanglefinal$ the angle at its apex. 

\STATE Define the cone
\begin{equation*}
  \cone_{\epcount} = \{x~|~\exists \lambda > 0,
  \alpha_1,\dots,\alpha_d > 0, \sum_{i=1}^d\alpha_i = 1~:~x = \apex -
  \lambda\sum_{i=1}^d\alpha_i(\base_i - \apex)\}
\end{equation*}

\STATE Set $\Ball^{'}_{\epcount+1}$ to be the min. volume ellipsoid
containing $\Ball_{\epcount}\setminus\cone_{\epcount}$. 

\STATE Set $\xset_{\epcount+1} = \xset_{\epcount} \cap
\Ball^{'}_{\epcount+1}$. 

\OUTPUT new feasible region $\xset_{\epcount+1}$ and enclosing
ellipsoid $\Ball^{'}_{\epcount+1}$. 

\end{algorithmic}
\label{alg:conecutting}
\end{algorithm}

\begin{algorithm}
\caption{\textsc{Hat-raising}}
\begin{algorithmic}[1]
\INPUT pyramid $\pyramid$ with apex $\apex$, (rounded) feasible region
$\xset_\epcount$ for epoch $\epcount$, enclosing ball
$\Ball_{\epcount}$. 

\STATE Let $\pyracenter$ be the center of $\pyramid$.

\STATE Set $\apex' = \apex + (\apex - \pyracenter)$. 

\STATE Set $\pyramid^{'}$ to be the pyramid with apex $\apex'$ and
same base as $\pyramid$.

\STATE Set $(\xset_{\epcount+1}, \Ball^{'}_{\epcount +1}) = $
\textsc{Cone-cutting}$(\pyramid^{'}, \xset_{\epcount}, \Ball_{\epcount})$. 

\OUTPUT new feasible region $\xset_{\epcount+1}$ and enclosing
ellipsoid $\Ball^{'}_{\epcount + 1}$. 

\end{algorithmic}
\label{alg:hatraising}
\end{algorithm}

Just like the 1-dimensional case, Algorithm~\ref{alg:highd} proceeds
in \emph{epochs}. We start with the optimization domain $\xset$, and
at the beginning we set $\xset_0 = \xset$. At the beginning of epoch
$\epcount$, we have a current feasible set $\xset_{\epcount}$ which
contains an approximate optimum of the convex function. The epoch ends
with discarding some portion of the set $\xset_{\epcount}$ in such a
way that we still retain at least one approximate optimum in the
remaining set $\xset_{\epcount+1}$.

At the start of the epoch $\epcount$, we apply an affine
transformation to $\xset_{\epcount}$ so that the smallest volume
ellipsoid containing it is a Euclidean ball of radius
$\outrad_{\epcount}$ (denoted as $\Ball(\outrad_{\epcount}))$. We
define $\inrad_{\epcount} = \outrad_{\epcount}/\inconst d$ for a
constant $\inconst \geq 1$, so that $\Ball(\inrad_{\epcount})
\subseteq \xset_{\epcount}$ (such a construction is always possible,
see, e.g., Lecture 1, p.~2~\cite{Ball97}). We will use the notation
$\Ball_{\epcount}$ to refer to the enclosing ball. Within each epoch,
the algorithm proceeds in several rounds, each round maintaining a
value $\gamma_{\roundctr}$ which is successively halved.

Let $\centerpt$ be the center of the ball $\ball(\outrad_{\epcount})$
containing $\xset_{\epcount}$. At the start of a round $i$, we
construct a regular simplex centered at $\centerpt$ and contained in
$\ball(\inrad_{\epcount})$. The algorithm queries the function $f$ at
all the vertices of the simplex, denoted by $x_1.\dots,x_{d+1}$, until
the CI's at each vertex shrink to $\gamma_{\roundctr}$. The algorithm
then picks the point $\apex_1$ for which the average of observed
function values is the largest. By construction, we are guaranteed
that $f(\apex_1) \geq f(x_j) - \gamma_{\roundctr}$ for all $j =
1,\dots,d+1$. This step is depicted in
Figure~\ref{fig:simplexsamples}.
\begin{figure}[htbp]
  \centering
  \includegraphics[height=2in]{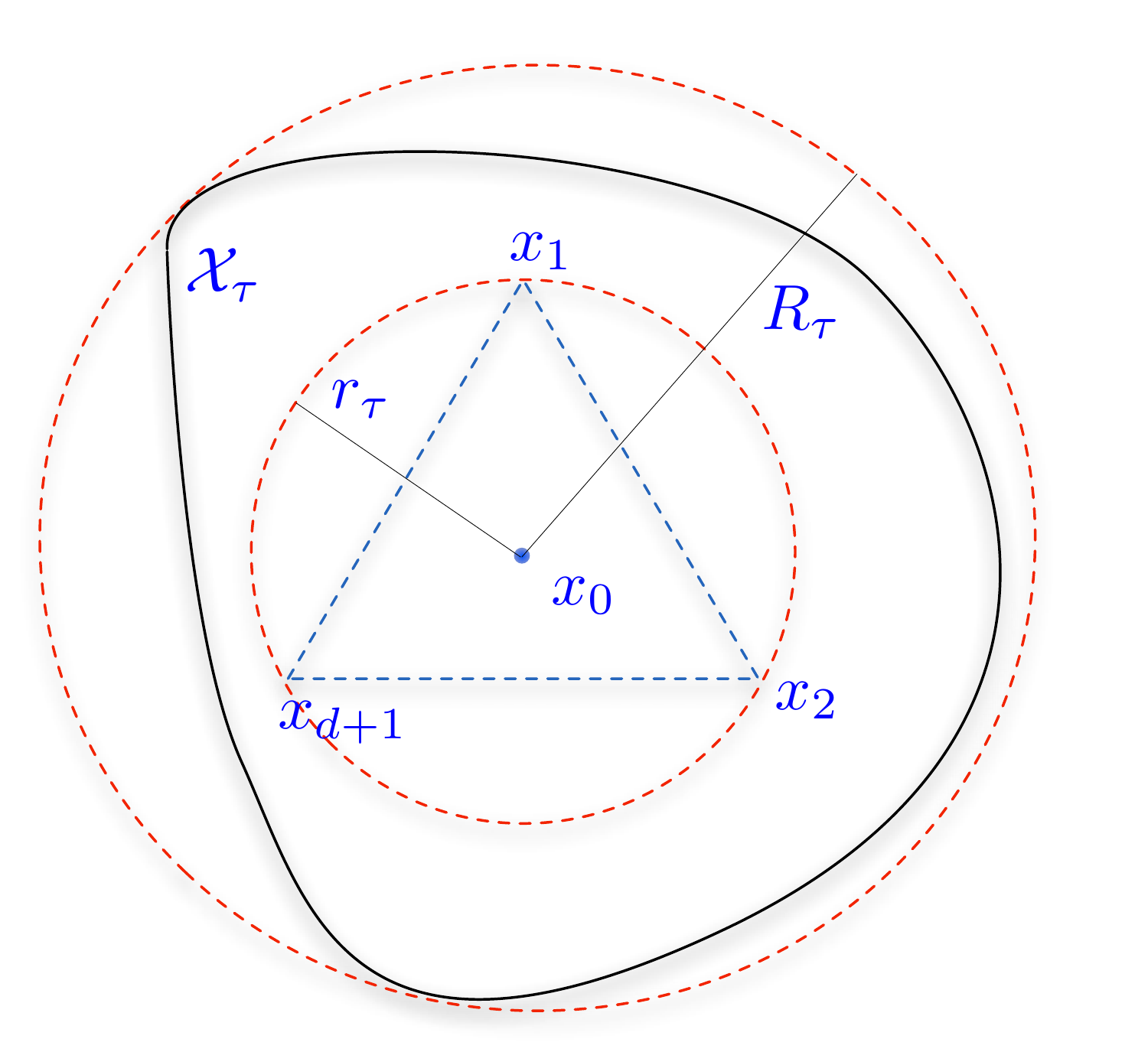}
  \caption{The regular simplex constructed at round $\roundctr$ of
    epoch $\epcount$ with radius $\inrad_{\epcount}$, center $\centerpt$
    and vertices $x_1,\dots,x_{d+1}$.}
  \label{fig:simplexsamples}
\end{figure}

\begin{figure}[t]
  \centering
  \includegraphics[height=2in]{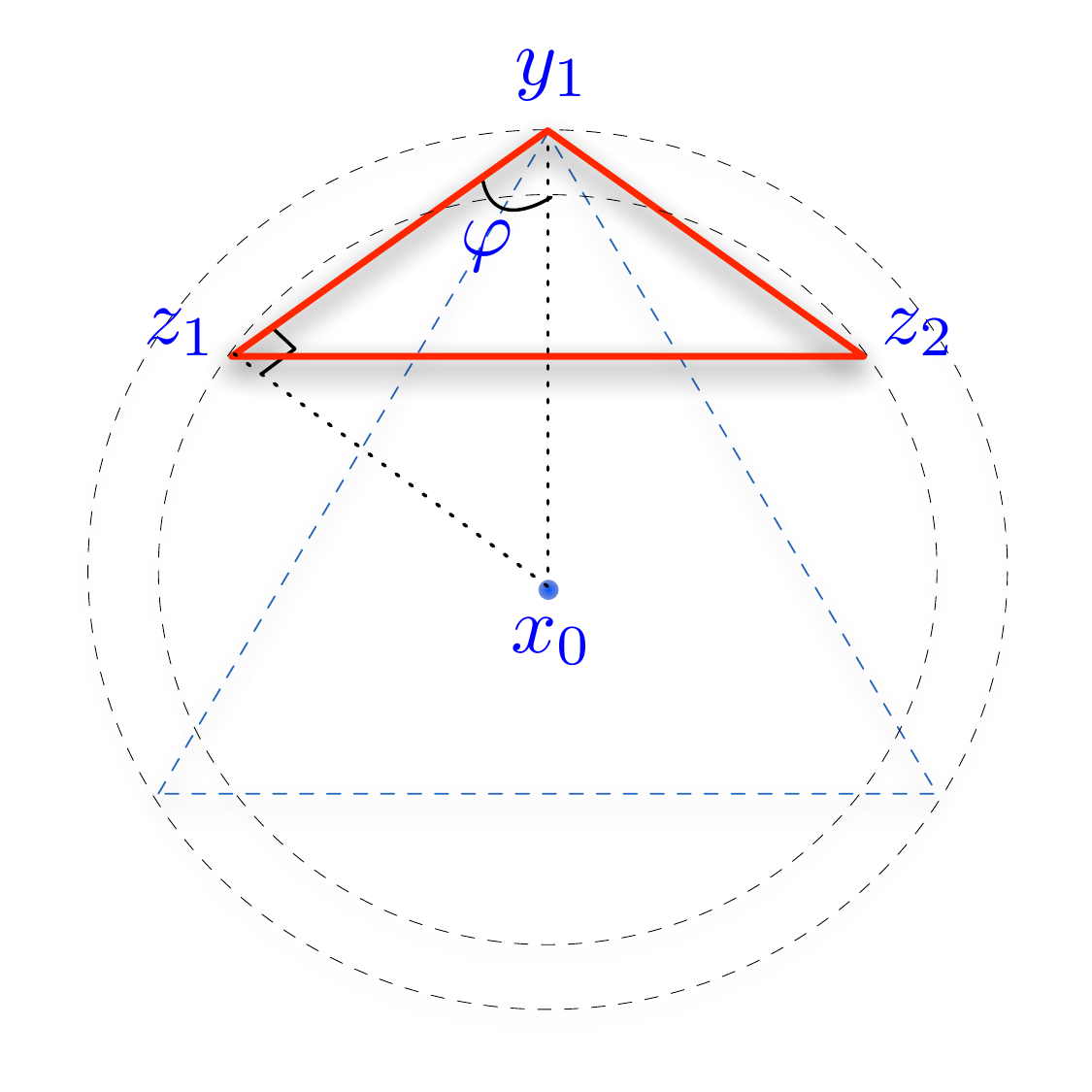}
  \includegraphics[height=2in]{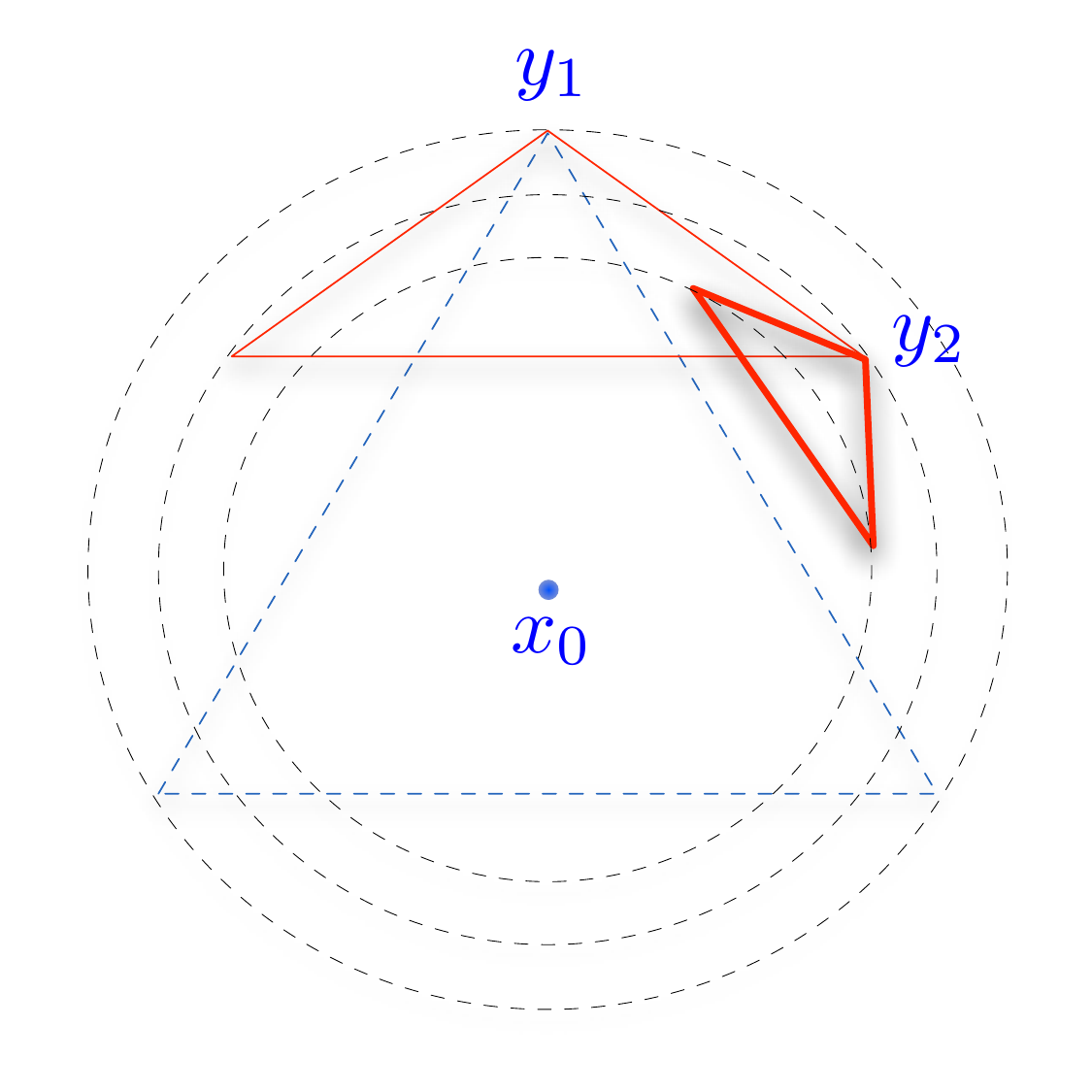}
  \includegraphics[height=2in]{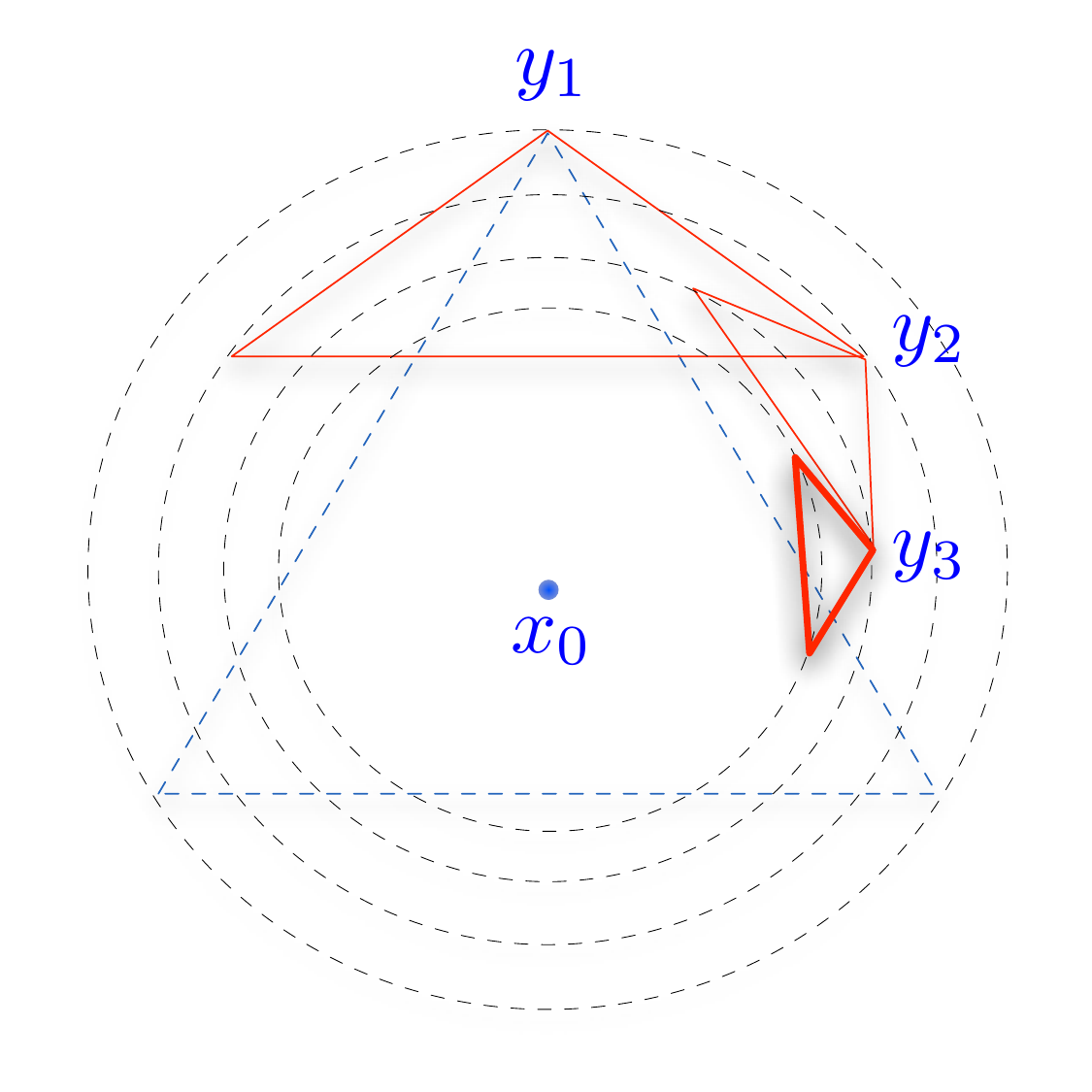}
  \caption{Pyramids constructed by Algorithm~\ref{alg:highd}. First
    diagram is the initial pyramid constructed by the algorithm at
    round $\roundctr$ of epoch $\epcount$ with apex $\apex_1$, base
    vertices $\base_1,\dots, \base_d$ and angle $\wideangle$ at the
    vertex. The other diagrams show the subsequent pyramids which
    successively get closer to the center of the ball}
  \label{fig:pyramid}
\end{figure}

The algorithm now successively constructs a sequence of pyramids, with
the goal of identifying a region of the feasible set
$\xset_{\epcount}$ such that at least one approximate optimum of $f$
lies outside the selected region. This region will be discarded at the
end of the epoch. The construction of the pyramids follows the
construction from Section 9.2.2 of the book~\cite{NemirovskiYu83}. The
pyramids we construct will have an angle $2\wideangle$ at the apex,
where $\cos\wideangle = \phiconst/d$. The base of the pyramid consists
of vertices $\base_1,\dots,\base_d$ such that $\base_i - \centerpt$
and $\apex_1 - \base_i$ are orthogonal. We note that the construction
of such a pyramid is always possible---we take a sphere with $\apex_1
- \centerpt$ as the diameter, and arrange $\base_1,\dots,\base_d$ on
the boundary of the sphere such that the angle between
$\apex_1-\centerpt$ and $\apex_1-\base_i$ is $\wideangle$. The
construction of the pyramid is depicted in
Figure~\ref{fig:pyramid}. Given this pyramid, we set $\hatgamma = 1$,
and sample the function at $\apex_1$ and $\base_1,\dots,\base_d$ as
well as the center of the pyramid until the CI's all shrink to
$\hatgamma$. Let $\top$ and $\bottom$ denote the vertices of the
pyramid (including $\apex_1$) with the largest and the smallest
function value estimates resp. For consistency, we will also use
$\apexsc$ to denote the apex $\apex_1$. We then check for
one of the following conditions:

\begin{enumerate}
\item If $\LB_\hatgamma(\top) \geq \UB_\hatgamma(\bottom) +
  \Delta_\epcount(\hatgamma)$, we proceed based on the separation
  between $\top$ and apex CI's as illustrated in
  Figures~\ref{fig:CIfig1}(a) and~\ref{fig:CIfig1}(b).
  \begin{figure}
    \centering
    \begin{tabular}{cc}
      \includegraphics[height=1.4in]{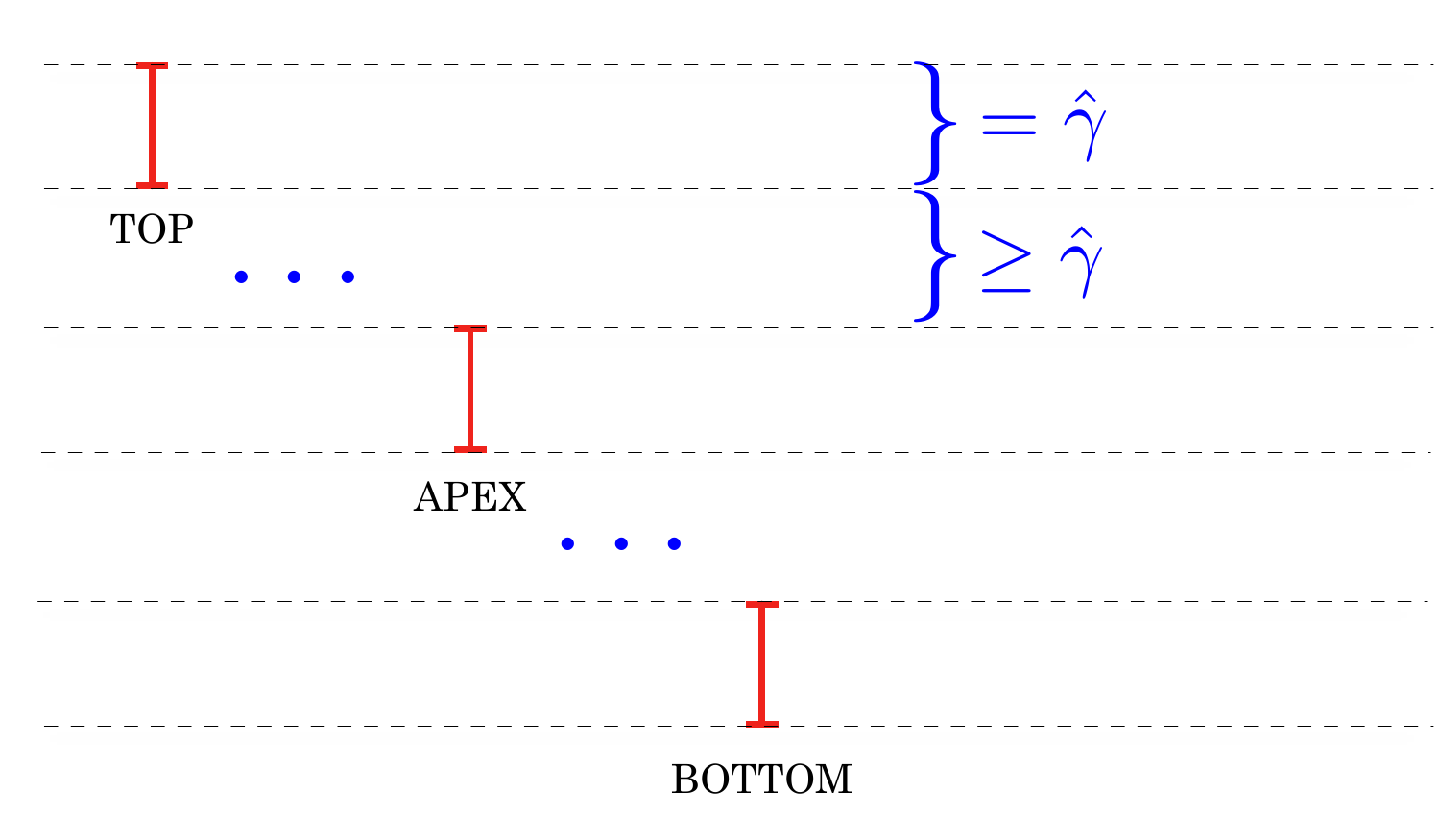} &
      \includegraphics[height=1.4in]{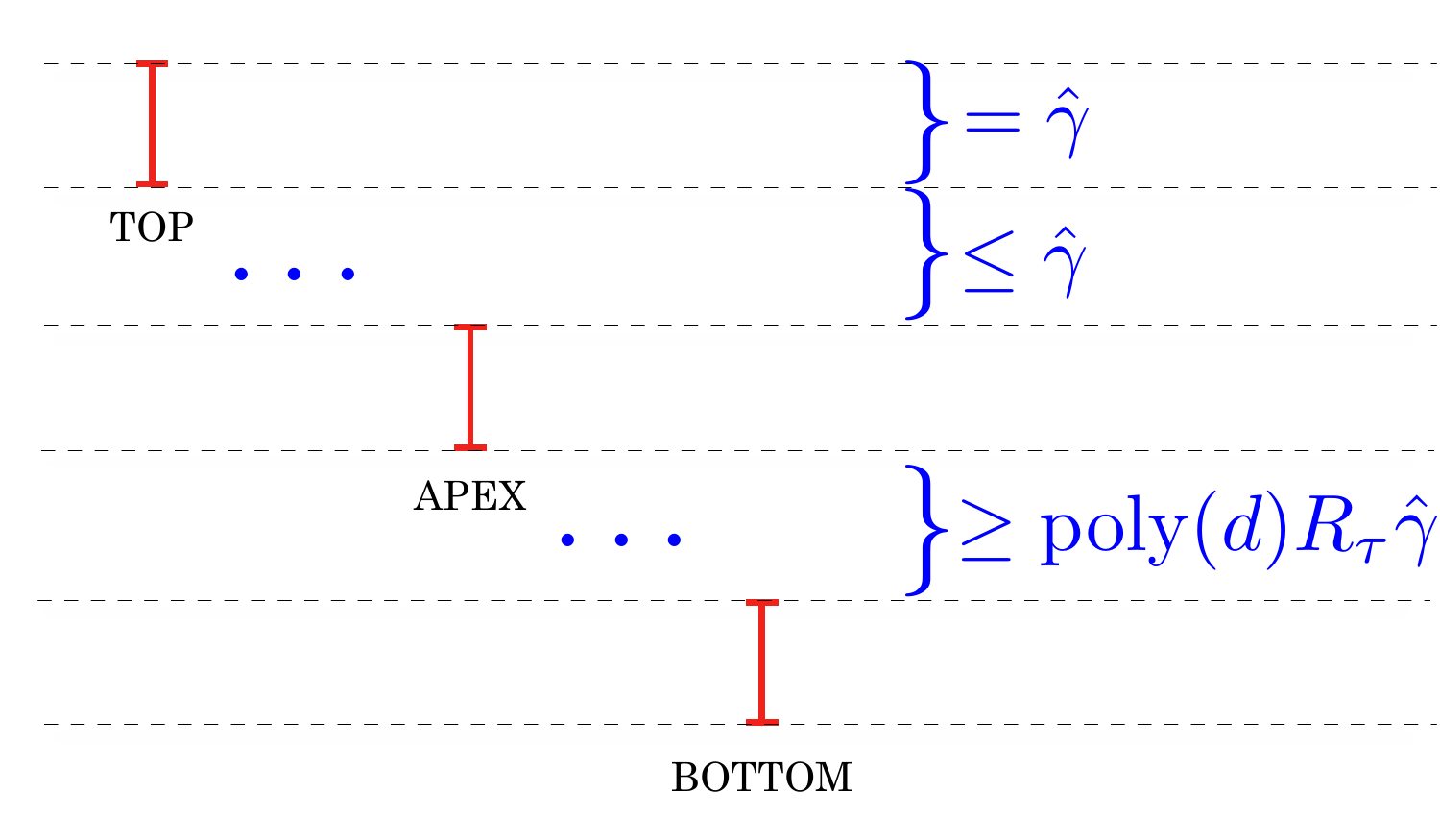}\\ 
      (a) & (b)
    \end{tabular}
    \caption{Relative ordering of confidence intervals of $\top$,
      $\bottom$ and $\apexsc$ in cases 1(a) and 1(b) of the
      algorithm resp.} 
    \label{fig:CIfig1}
  \end{figure}
  
  \begin{enumerate}
  \item If $\LB_\hatgamma(\top) \geq \UB_\hatgamma(\apexsc) +
    \hatgamma$, then we know that with high probability

    \begin{equation}
      \label{eqn:topgammamax}
      f(\top) \geq f(\apexsc) + \hatgamma \geq f(\apexsc) + \gamma_i.
    \end{equation}
    In this case, we set $\top$ to be the apex of the next pyramid,
    reset $\hatgamma = 1$ and continue the sampling procedure on the
    next pyramid. 

  \item If $\LB_\hatgamma(\top) \leq \UB_\hatgamma(\apexsc) +
    \hatgamma$, then we know that $\LB_\hatgamma(\apexsc) \geq
    \UB_\hatgamma(\bottom) + \Delta_{\epcount}(\hatgamma) -
    2\hatgamma$. In this case, we declare the epoch over and pass the
    current apex to the cone-cutting step.
  \end{enumerate}

\item If $\LB_\hatgamma(\top) \leq \UB_\hatgamma(\bottom) +
  \Delta_\epcount(\hatgamma)$, then one of the two events depicted in
  Figures~\ref{fig:CIfig2}(a) or~\ref{fig:CIfig2}(b) has to happen:
  \begin{figure}
	\label{fig:CIfig2}
    \centering
    \begin{tabular}{cc}
      \includegraphics[height=1.4in]{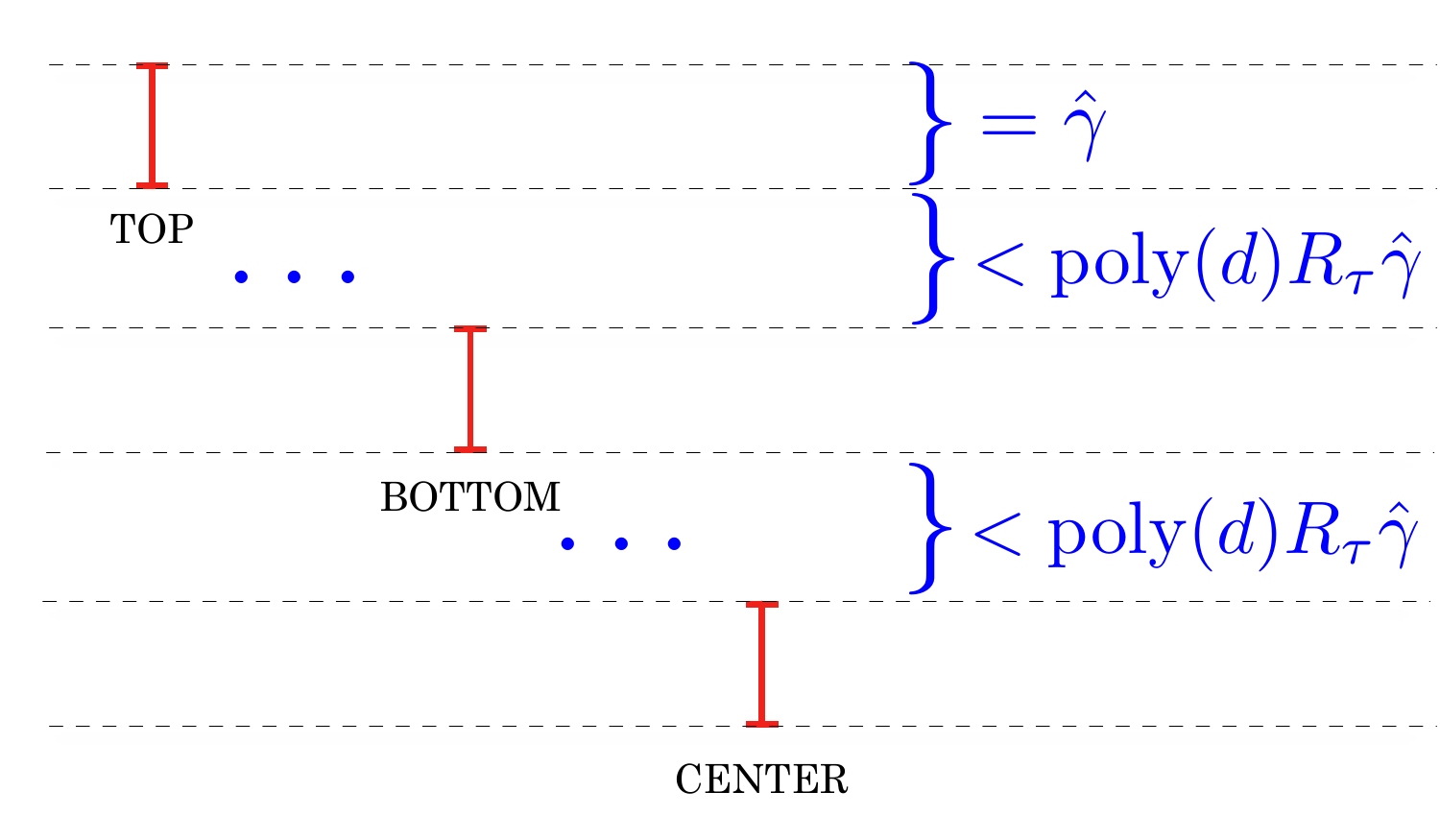} &
      \includegraphics[height=1.4in]{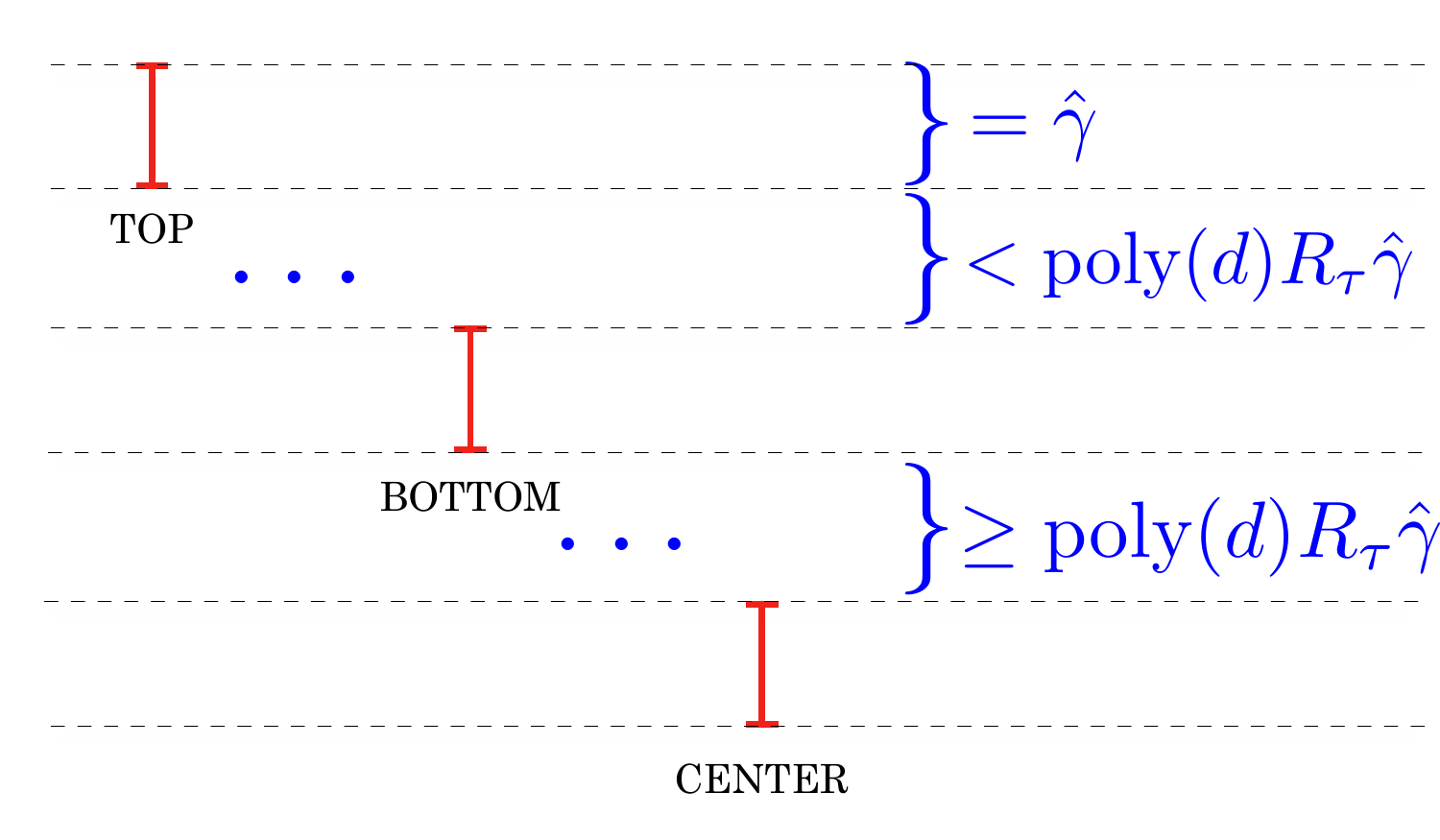}\\ 
      (a) & (b)
    \end{tabular}
    \caption{Relative ordering of confidence intervals of $\top$,
      $\bottom$ and center $\pyracenter$ in cases 2(a) and 2(b) of
      the algorithm resp.}
  \end{figure}

  \begin{enumerate}
  \item If $\UB_\hatgamma(\pyracenter) \geq \LB_\hatgamma(\bottom) -
    \otherdelta_{\epcount}(\hatgamma)$, then all of the vertices and the
    center of the pyramid have their function values within a
    $2\Delta_{\epcount}(\hatgamma) + 3\hatgamma$ interval. In this
    case, we set $\hatgamma = \hatgamma/2$. If this sets $\hatgamma <
    \gamma_{\roundctr}$, we start the next round with
    $\gamma_{\roundctr+1} = \gamma_{\roundctr}/2$. Otherwise, we
    continue sampling the current pyramid with the new value of
    $\hatgamma$.

  \item If $\UB_\hatgamma(\pyracenter) \leq \LB_\hatgamma(\bottom) -
    \otherdelta_{\epcount}(\hatgamma)$, then we terminate the epoch and pass
    the center and the current apex to the hat-raising step.
  \end{enumerate}
\end{enumerate}

\paragraph{\textsc{Hat-Raising}:} This step happens when we construct a pyramid
where $\LB_\hatgamma(\top) \leq \UB_\hatgamma(\bottom) +
\Delta_\epcount(\hatgamma)$ but $\UB_\hatgamma(\pyracenter) \leq
\LB_\hatgamma(\bottom) - \otherdelta_{\epcount}(\hatgamma)$ (see
Fig.~\ref{fig:CIfig2}(b) for an illustration). In this case, we will
show that if we move the apex of the pyramid a little from
$\apex_{\roundctr}$ to $\apex^{'}_{\roundctr}$, then
$\apex^{'}_{\roundctr}$'s CI is above the $\top$ CI while the angle of
the new pyramid at $\apex^{'}_{\roundctr}$ is not much smaller than
$2\wideangle$. In particular, letting $\pyracenter_\roundctr$ denote
the center of the pyramid, we set $\apex^{'}_{\roundctr} =
\apex_{\roundctr} + (\apex_{\roundctr} -
\pyracenter_\roundctr)$. Figure~\ref{fig:hatraising} shows
transformation of the pyramid involved in this step. The correctness
of this step and the sufficiency of the perturbation from $\apex$ to
$\apex^{'}$ will be proved in the next section.

  \begin{figure}
    \centering
      \includegraphics[height=1.4in]{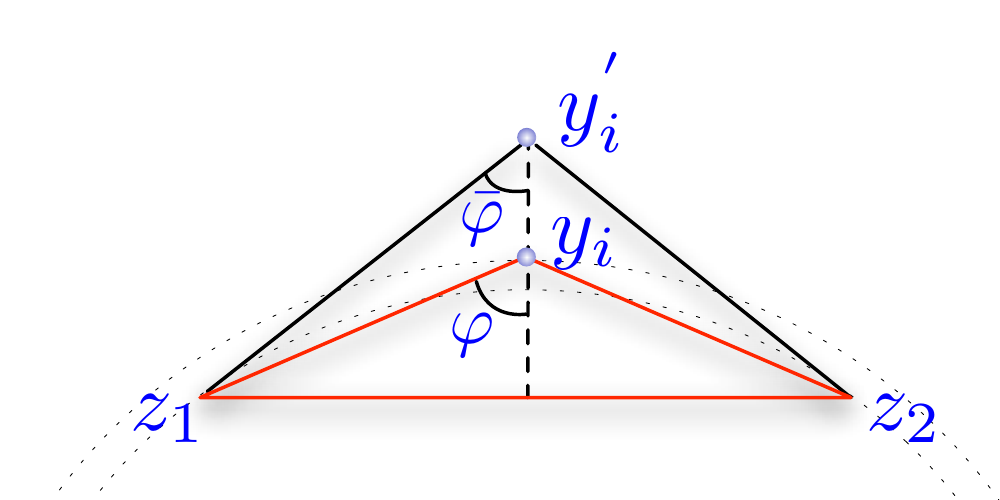}
    \caption{Transformation of the pyramid $\pyramid$ in the
      hat-raising step.}
    \label{fig:hatraising}
  \end{figure}

\paragraph{\textsc{Cone-cutting}:} This step is the concluding step for an
epoch. The algorithm gets to this step either through case 1(b) or
through the hat-raising step. In either case, we have a pyramid with
an apex $\apex$, base $\base_1,\dots,\base_d$ and an angle
$2\wideanglefinal$ at the apex, where $\cos(\wideanglefinal) \leq
1/2d$. We now define a cone

\begin{equation}
  \cone_{\epcount} = \{x~|~\exists \lambda > 0,
  \alpha_1,\dots,\alpha_d > 0, \sum_{i=1}^d\alpha_i = 1~:~x = \apex -
  \lambda\sum_{i=1}^d\alpha_i(\base_i - \apex)\}
  \label{eqn:wideanglecone}
\end{equation}
which is centered at $\apex$ and a reflection of the pyramid around
the apex. By construction, the cone $\cone_{\epcount}$ has an angle
$2\wideanglefinal$ at its apex. We set $\Ball_{\epcount+1}^{'}$ to be
the ellipsoid of minimum volume containing
$\Ball_{\epcount}\setminus\cone_{\epcount}$ and define $\xset_{\epcount+1} =
\xset_{\epcount}\cap \Ball_{\epcount+1}^{'}$. This is illustrated in
Figure~\ref{fig:conecutting}. Finally, we put things back into an
isotropic position and $\Ball_{\epcount+1}$ is the ball containing
$\xset_{\epcount+1}$ is in the isotropic coordinates, which is just
obtained by applying an affine transformation to
$\Ball_{\epcount+1}^{'}$.

\begin{figure}
  \centering
  \includegraphics[height=1.5in]{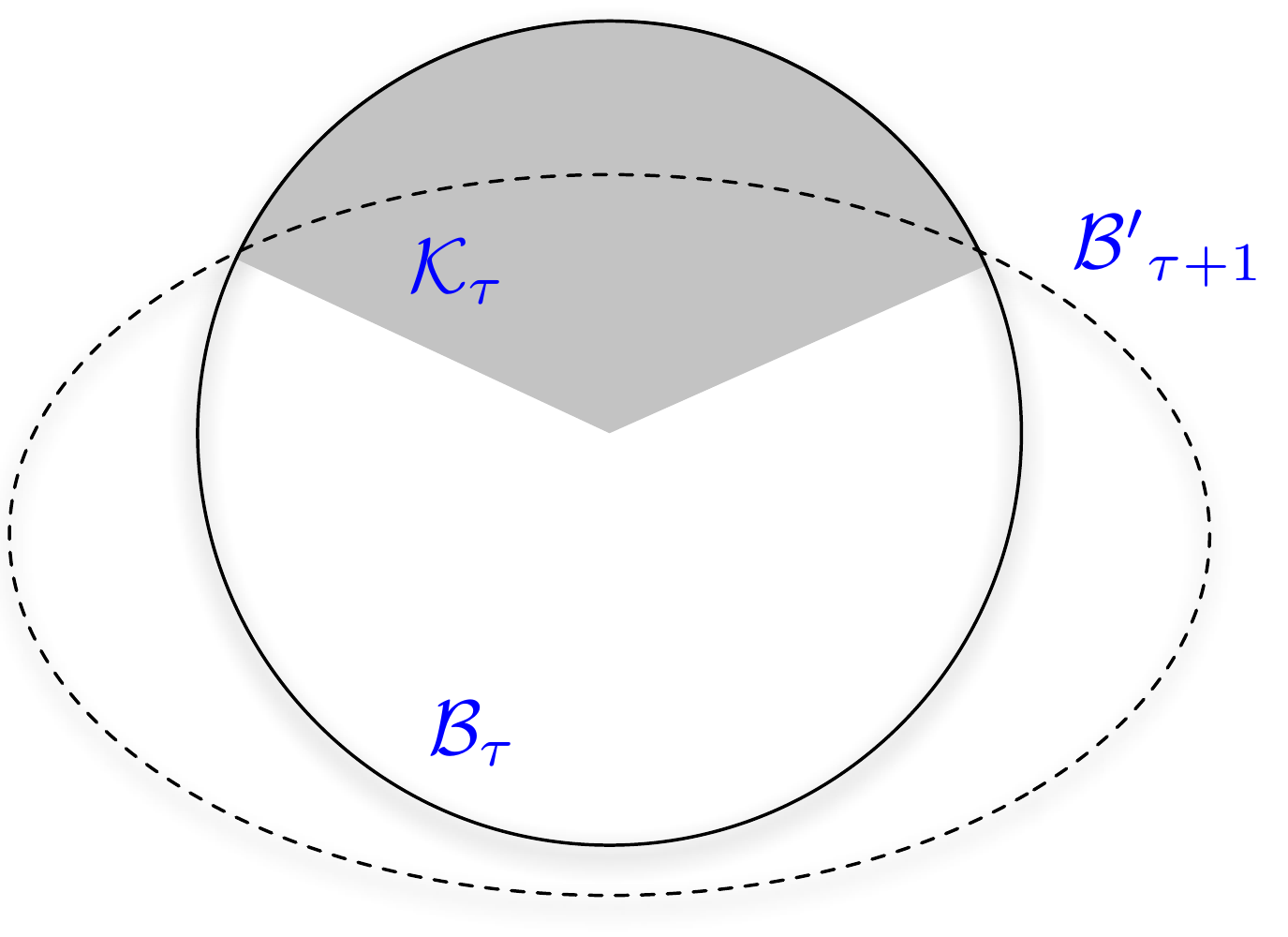}
  \caption{Illustration of the cone-cutting step at epoch
    $\epcount$. Solid circle is the enclosing ball
    $\Ball_{\epcount}$. Shaded region is the intersection of
    $\cone_{\epcount}$ with $\Ball_{\epcount}$. The dotted ellipsoid
    is the new enclosing ellipsoid $\Ball^{'}_{\epcount+1}$ for the
    residual domain.} 
  \label{fig:conecutting}
\end{figure}

Let us end the description with a brief discussion regarding the
computational aspects of this algorithm. It is clear that the most
computationally intensive steps of this algorithm are the cone-cutting
and isotropic transformation at the end. However, these steps are
exactly analogous to an implementation of the classical ellipsoid
method. In particular, the equation for $\Ball^{'}_{\epcount+1}$ is
known in closed form~\cite{GoldfarbTodd82}. Furthermore, the affine
transformations needed to the reshape the set can be computed via
rank-one matrix updates and hence computation of inverses can be done
efficiently as well (see e.g.~\cite{GoldfarbTodd82} for the relevant
implementation details of the ellipsoid method). 

\section{Analysis}

We start by showing the correctness of the algorithm and then proceed
to regret analysis. To avoid having probabilities throughout our
analysis, we define an event $\event$ where at each epoch $\epcount$,
and each round $i$, $f(x) \in [\LB_{\gamma_i}(x),\UB_{\gamma_i}(x)]$
for any point $x$ sampled in the round. We will carry out the
remainder of the analysis conditioned on $\event$ and bound the
probability of $\event^c$ at the end. We also assume that the
algorithm is run with the settings

\begin{equation}
  \Delta_{\epcount}(\gamma) = \left(\frac{6\inconst
    d^4}{\phiconst^2} + 3\right)\gamma \quad\mbox{and}\quad 
  \otherdelta_{\epcount}(\gamma) = \left(\frac{6\inconst 
    d^4}{\phiconst^2} + 5\right)\gamma,
\label{eqn:Deltas}
\end{equation}
and constants $\inconst \geq 64$, $\phiconst \leq 32$.

\subsection{Correctness of the algorithm}

In order to complete the proof of our algorithm's correctness, we only
need to further show that when the algorithm proceeds to cone-cutting
via case 1(b), then it does not discard all the approximate optima of
$f$ by mistake, and show that the hat-raising step is indeed correct
as claimed. These two claims are established in the next couple of
lemmas.

For these two lemmas, we assume that the distance of the apex of any
$\pyramid$ constructed in epoch $\epcount$ from the center of
$\ball(\inrad_\epcount)$ is at least $\inrad_\epcount / d$.
This assumption will be established later.

\begin{lemma}
  Let $\cone_{\epcount}$ be the cone discarded at epoch $\epcount$
  which is ended through Case (1b) in round $i$.  Let $\bottom$ be the
  lowest CI of the last pyramid $\pyramid$ constructed in the epoch,
  and assume the distance from the apex of $\pyramid$ to the center of
  $\ball(\inrad_\epcount)$ is at least $\inrad_\epcount / d$.  Then
  $f(x) \geq f(\bottom) + \gamma_i$ for all $x \in \cone_{\epcount}$.
  \label{lemma:correct1b}
  \end{lemma}

\begin{proof}
  Consider any $x \in \cone_{\epcount}$. By construction, there is a
  point $\base$ in the base of the pyramid $\pyramid$ such that the
  apex $\apex$ of $\pyramid$ satisfies $\apex = \alpha\base +
  (1-\alpha)x$ for some $\alpha \in [0,1)$ (see
    Fig.~\ref{fig:correctness1bfunctionvals} for a graphical
    illustration).
    \begin{figure}
      \centering
      \includegraphics[height=1.4in]{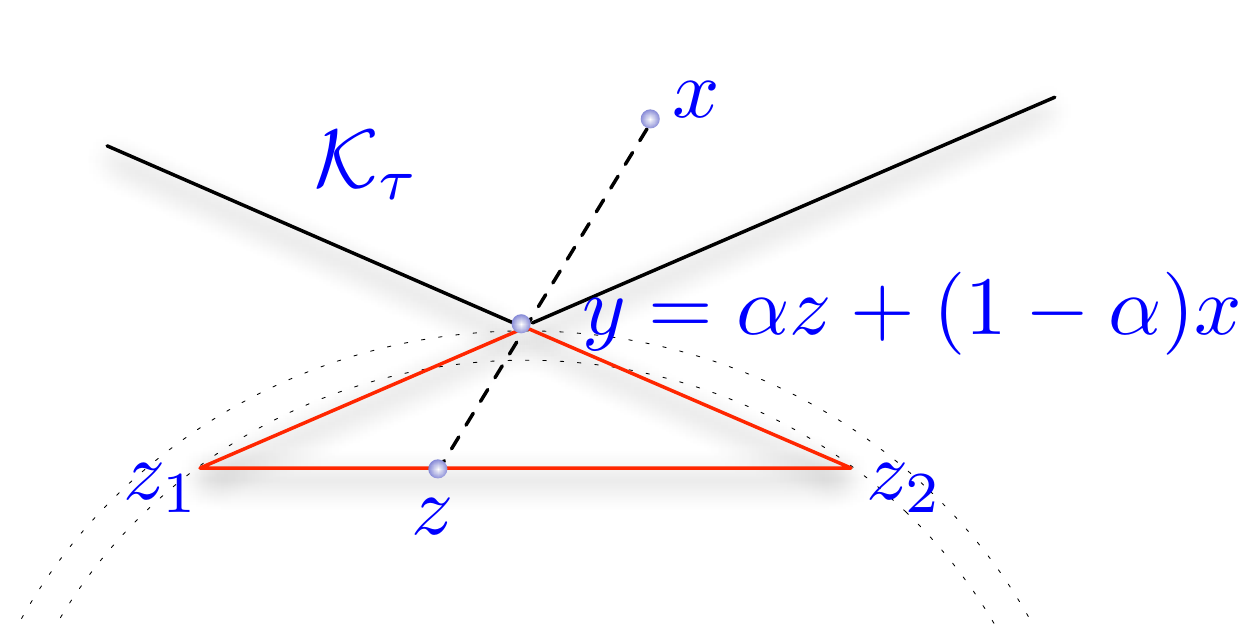}
      \caption{The points of interest in Lemma~\ref{lemma:correct1b}
        (see text). Solid lines depict the pyramid $\pyramid$ and the
        $\cone_{\epcount}$.}
      \label{fig:correctness1bfunctionvals}
    \end{figure}

  Since $f$ is convex and $\base$ is in the base of the pyramid, we have
  that 
  \[ f(\base) \leq f(\top) \leq f(\apex) + 3\hatgamma \].
  Also, the condition of Case (1b) ensures
  \[ f(\apex) > f(\bottom) + \Delta_\epcount(\hatgamma) - 2\hatgamma \]
  where $\hatgamma$ is the CI level used for the pyramid.
  Then by convexity of $f$
  \begin{align*}
    f(\apex) \leq \alpha f(\base) + (1-\alpha) f(x) \leq \alpha
    (f(\apex)  + 3\hatgamma) + (1-\alpha) f(x).
  \end{align*}
  Simplifying yields
  \[ f(x) \geq f(\apex) - 3\frac{\alpha}{1-\alpha} \hatgamma
  > f(\bottom) + \Delta_\epcount(\hatgamma) - 2\hatgamma -
  3\frac{\alpha}{1-\alpha} \hatgamma . \] Also, we know that
  $\alpha/(1-\alpha) = \|\apex - x\|/\|\apex - \base\|$.  Because $x
  \in \ball(\outrad_\epcount)$, $\|\apex - x\| \leq 2\outrad_\epcount
  \leq 2\inconst d\inrad_\epcount$.  Moreover, $\|\apex - \base\|$ is
  at least the height of $\pyramid$, which is at least
  $\inrad_\epcount\phiconst^2/d^3$ by
  Lemma~\ref{lemma:pyramid-construction}.  Therefore
  \begin{equation*}
    \frac{\alpha}{1-\alpha} = \frac{\|\apex - x\|}{\|\apex - \base\|}
    \leq \frac{2\inconst d\inrad_\epcount}{\inrad_\epcount\phiconst^2/d^3}
    \leq \frac{2\inconst d^4}{\phiconst^2}.
  \end{equation*}
  Thus, we have 
  \begin{align}
  f(x)
  & > f(\bottom) + \Delta_\epcount(\hatgamma) - 2\hatgamma -
  \frac{6\inconst d^4}{\phiconst^2} \hatgamma
  \geq f(\bottom) + \gamma_i
  ,
  \label{eqn:regret1b}
  \end{align}
  where the last line uses the setting of
  $\Delta_{\epcount}(\hatgamma)$~\eqref{eqn:Deltas}, completing the
  proof of the lemma.
\end{proof}

This lemma guarantees that we cannot discard all the approximate
minima of $f$ by mistake in case 1(b), and that any point discarded by
the algorithm through this step in round $i$ has regret at least
$\gamma_i$. The final check that needs to be done is the correctness
of the hat-raising step which we do in the next lemma.

\begin{lemma}
  Let $\pyramid'$ be the new pyramid formed in hat-raising with apex
  $\apex'$ and same base as $\pyramid$ in round $i$ of epoch
  $\epcount$, and let $\cone_\epcount'$ be the cone discarded. Assume
  the distance from the apex of $\pyramid$ to the center of
  $\ball(\inrad_\epcount)$ is at least $\inrad_\epcount / d$.  Then
  the $\pyramid'$ has an angle $\wideanglefinal$ at the apex with
  $\cos\wideanglefinal \leq 2\phiconst/d$, height at most
  $2\inrad_\epcount\inconst^2/d^2$, and with every point $x$ in the
  cone $\cone_\epcount'$ having $f(x) \geq f(x^*) + \gamma_i$.
  \label{lemma:correcthatraising}
\end{lemma}

\begin{proof}
Let $\apex' := \apex + (\apex - \pyracenter)$ be the apex of $\pyramid'$.
Let $h$ be the height of $\pyramid$ (the distance from $\apex$ to the
base), $h'$ be the height of $\pyramid'$, and $b$ be the distance from any
vertex of the base to the center of the base.
Then $h' < 2h \leq 2\inrad_\epcount\inconst^2/d^2$ by
Lemma~\ref{lemma:pyramid-construction}.
Moreover, since $\cos(\wideangle) = h / \sqrt{h^2+b^2} = 1/d$, we have
$\cos(\wideanglefinal) = h' / \sqrt{{h'}^2+b^2} \leq 2h / \sqrt{h^2+b^2} =
2\cos(\wideangle) = 2\phiconst/d$.

It remains to show that every $x \in \cone_\epcount'$ has $f(x) \geq f(x^*)
+ \hatgamma$.
By convexity of $f$, $f(\apex) \leq (f(\apex') + f(\pyracenter))/2$, so
$f(\apex') \geq 2f(\apex) - f(\pyracenter)$.
Since we enter hat-raising via case 2(b) of the algorithm, we know that
$f(\pyracenter) \leq f(\apex) - \otherdelta_{\epcount}(\hatgamma)$, so

\[ f(\apex') \geq f(\apex) + \otherdelta_{\epcount}(\hatgamma) . \]
The condition for entering case 2(b) also implies that $f(\apex) >
f(\top) - \Delta_\epcount(\hatgamma) - 2\hatgamma> f(x) -
\Delta_\epcount(\hatgamma) - 2\hatgamma$ for all $x \in \pyramid$, and
therefore for any $\base$ on the base of $\pyramid$,

\[ f(\apex') > f(\base) + \otherdelta_{\epcount}(\hatgamma) -
\Delta_\epcount(\hatgamma) - 2\hatgamma \geq f(\base) , \] 
where the last line uses the settings of
$\Delta_{\epcount}(\hatgamma)$ and
$\otherdelta_{\epcount}(\hatgamma)$~\eqref{eqn:Deltas}. Now take any
$x \in \cone_\epcount'$.  There exists $\alpha \in [0,1)$ and $\base$
  on the base of $\pyramid'$ such that $\apex' = \alpha \base +
  (1-\alpha) x$, so by convexity of $f$, $f(\apex') \leq \alpha
  f(\base) + (1-\alpha) f(x) \leq \alpha f(\apex') + (1-\alpha) f(x)$,
  which implies $f(x) \geq f(\apex') \geq f(\apex) +
  \otherdelta_{\epcount}(\hatgamma) \geq f(x^*) + \gamma_i$.
\end{proof}

\subsection{Regret analysis}

The following theorem states our regret guarantee on the performance
of the algorithm~\ref{alg:highd}.

\begin{theorem}
  Suppose Algorithm~\ref{alg:highd} is run with $\inconst \geq 64$,
  $\phiconst \leq 1/32$ and parameters
  \begin{equation*}
    \Delta_{\epcount}(\gamma) = \left(\frac{6\inconst
      d^4}{\phiconst^2} + 3\right)\gamma \quad\mbox{and}\quad 
    \otherdelta_{\epcount}(\gamma) = \left(\frac{6\inconst 
      d^4}{\phiconst^2} + 5\right)\gamma.
  \end{equation*}
  Then with probability at least $1-1/T$, the net regret incurred by
  the algorithm is bounded by 
  \begin{equation*}
    768d^3\sigma\sqrt{T}\log^2 T \left(\frac{2d^2\log d}{\phiconst^2}
    + 1\right)\left(\frac{4d^7\inconst}{\phiconst^3} +
    \frac{d(d+1)}{\phiconst} \right) \left(\frac{4\inconst
      d^4}{\phiconst^2} + 11\right).
  \end{equation*}
  \label{thm:highd}
\end{theorem}

\paragraph{Remarks:} The prior knowledge of $T$ in
Algorithm~\ref{alg:highd} and Theorem~\ref{thm:highd} can again be
addressed using a doubling argument. As earlier,
Theorem~\ref{thm:highd} is optimal in the dependence on $T$. The large
dependence on $d$ is also seen in Nemirovski and Yudin
\cite{NemirovskiYu83} who obtain a $d^7$ scaling in noiseless case and
leave it an unspecified polynomial in the noisy case. Using random
walk ideas~\cite{Bertsimas2004} to improve the dependence on $d$ is an
interesting question for future research.

The analysis will start by controlling the regret incurred on
different rounds, and then we will piece it together across rounds and
epochs to get the net regret for the entire procedure.

\subsubsection{Bounding the regret incurred in one round}

We will start by a simple lemma regarding the regret incurred while
playing a pyramid if the condition 2(a) is encountered in the
algorithm. This lemma highlights the importance of evaluating the
function at the center of the pyramid, a step that was not needed in
the framework of Nemirovski and Yudin~\cite{NemirovskiYu83}. We will
use the symbol $\pyramid$ to refer to a generic pyramid constructed by
the algorithm during the course of its operation, with apex $\apex$,
base $\base_1,\dots, \base_d$, center $\pyracenter$ and with an angle
$\wideangle$ at the apex. We also recall that the pyramids constructed
by the algorithm are such that the distance from the center to the
base is at least $\inrad_{\epcount}\phiconst^2/d^3$.

\begin{lemma}
Suppose the algorithm reaches case 2(a) in round $i$ of epoch
$\epcount$, and assume $x^* \in \ball(\outrad_\epcount)$ where $x^*$
is the minimizer of $f$.  Let $\pyramid$ be the current pyramid and
$\hatgamma$ be the current CI width. Assume the distance from the apex
of $\pyramid$ to the center of $\ball(\inrad_\epcount)$ is at least
$\inrad_\epcount / d$. Then the net regret incurred while evaluating
the function on $\pyramid$ in round $i$ is at most
\[
\frac{6d\sigma\log T}{\hatgamma}\left(\frac{4d^7\inconst}{\phiconst^3}
+ \frac{d(d+2)}{\phiconst} \right) \left(\frac{12\inconst
  d^4}{\phiconst^2} + 11\right) .
\]
\label{lemma:functionflat}
\end{lemma}
\begin{proof}
The proof is a consequence of convexity. We start by bounding the
variation of the function inside the pyramid.  Since the pyramid is a
convex hull of its vertices, we know that the function value at any
point in the pyramid is also upper bounded by the largest function
value achieved at any vertex.  Furthermore, the condition for reaching
Case (2a) implies that the function value at any vertex is at most
$f(\pyracenter) + \Delta_\epcount(\hatgamma) +
\otherdelta_{\epcount}(\hatgamma) + 3\hatgamma$, and therefore
\begin{equation}
    f(x) \leq f(\pyracenter) + \Delta_\epcount(\hatgamma) +
    \otherdelta_{\epcount}(\hatgamma) + 3\hatgamma \quad \text{for all $x
      \in \pyramid$.}
    \label{eqn:pyramidfupperbound}
\end{equation}
For brevity, we use the shorthand $\delta :=
\Delta_\epcount(\hatgamma) + \otherdelta_{\epcount}(\hatgamma) + 3\hatgamma$.
Consider any point $x \in \pyramid$, and let $b$ be the point where the ray
$\pyracenter - x$ intersects a face of $\pyramid$ on the other side.
Then we know that there is a positive constant $\alpha \in [0,1]$ such that
$\pyracenter = \alpha x + (1-\alpha) b$; in particular,
$(1-\alpha) / \alpha = \|\pyracenter - x\| / \|\pyracenter - b\|$.
Note that $\|\pyracenter-x\|$ is at most the distance from $\pyracenter$ to
a vertex of $\pyramid$, and $\|\pyracenter-b\|$ is at least the radius of
the largest ball centered at $\pyracenter$ inscribed in $\pyramid$.
Therefore by Lemma~\ref{lemma:pyramid-inscribed}(b),
\begin{equation*}
\frac{1-\alpha}{\alpha}
= \frac{\|\pyracenter - x\|}{\|\pyracenter - b\|}
\leq \frac{d(d+1)}{\phiconst}
.
\end{equation*}
Then the convexity of $f$ and the upper bound on function values over
$\pyramid$ from~\eqref{eqn:pyramidfupperbound} guarantee that
\begin{equation*}
  f(\pyracenter) \leq \alpha f(x) + (1-\alpha) f(b) \leq \alpha
  f(x) + (1-\alpha) (f(\pyracenter) + \delta).
\end{equation*}
Rearranging, we get 
\begin{align}
  f(x) \geq f(\pyracenter) - \frac{d(d+1)\delta}{\phiconst} .
\label{eqn:pyramidflowerbound}
\end{align}
  Combining equations~\eqref{eqn:pyramidfupperbound}
  and~\eqref{eqn:pyramidflowerbound} we have shown that for any $x,x' \in
  \pyramid$ 
  \begin{equation}
    |f(x) - f(x')| \leq  \frac{d(d+2)\delta}{\phiconst}.
    \label{eqn:pyramidfflat}
  \end{equation}

  Now we will bootstrap to show that the above bound implies low
  regret while sampling the vertices and center of $\pyramid$.  We
  first note that if $x^* \in \pyramid$, then the regret on any vertex
  or the center is bounded by $d(d+2)\delta/\phiconst$.  In that case,
  the regret incurred by sampling the vertices and center of this
  pyramid (so $d+2$ points) is bounded by $(d+2) \cdot
  d(d+2)\delta/\phiconst$.  Furthermore, we only need to sample each
  point pyramid $2\sigma\log T/{\hatgamma}^2$ times to get the CI's of
  width $\hatgamma$, which completes the proof in this case, so the
  total regret incurred is
  \[ (d+2)\frac{d(d+2)\delta}{\phiconst} \cdot \frac{2\sigma\log T}{\hatgamma^2} . \]

Now we consider the case where $x^* \notin \pyramid$. Recall that
Lemma~\ref{lemma:correct1b} guarantees that $x^* \in
\Ball_{\epcount}$.  There is a point $b$ on a face of $\pyramid$ such
that $b = \alpha x^* + (1-\alpha) \pyracenter$ for some $\alpha \in
[0,1]$.  Then $\alpha = \|\pyracenter-b\| / \|\pyracenter-x^*\|$.  By
the triangle inequality, $\|\pyracenter-x^*\| \leq 2\outrad_\epcount =
2\inconst d\inrad_\epcount$.  Moreover, $\|\pyracenter-b\|$ is at
least the radius of the largest ball centered at $\pyracenter$
inscribed in $\pyramid$, which is at least $\inrad_\epcount\phiconst^2
/ (2d^4)$ by Lemma~\ref{lemma:pyramid-inscribed}.  Therefore $\alpha
\geq \phiconst^2/(4\inconst d^5)$.  By convexity and
Equation~\eqref{eqn:pyramidflowerbound},
\[ f(\pyracenter) - \frac{d(d+2)\delta}{\phiconst} \leq f(b) \leq
\alpha f(x^*) + (1-\alpha) f(\pyracenter) , \] 
so
\[ f(x^*) \geq f(\pyracenter) - \frac{d(d+2)\delta}{\phiconst\alpha}
\geq f(\pyracenter) - \frac{4d^7 \inconst \delta}{\phiconst^3} \geq
f(x) - \frac{4d^7 \inconst \delta}{\phiconst^3} -
\frac{d(d+2)\delta}{\phiconst} 
\]
for any $x \in \pyramid$.
Therefore, using the same argument as before, the net regret incurred in
the round is
\[ (d+2) \left(\frac{4d^7\inconst}{\phiconst^3} +
\frac{d(d+2)}{\phiconst} \right) \delta \cdot \frac{2\sigma\log
  T}{\hatgamma^2}. 
 \]
Substituting in the values of $\Delta_{\epcount}(\hatgamma)$ and
$\otherdelta_{\epcount}(\hatgamma)$ completes the proof. 
\end{proof}

Lemma~\ref{lemma:functionflat} is critical because it allows us to
claim that at any round, when we sample the function over a pyramid
with a value $\hatgamma$, then the regret on that pyramid during this
sampling is at most $\poly(d)/\hatgamma$  
since we must have been in case 2(a) with $2\hatgamma$ if we're using
$\hatgamma$. The only exception is at first round, where this
statement holds trivially as the function is 1-Lipschitz by
assumption. 

We next show that the algorithm can visit the case 1(a) only a bounded
number of times every round. The round is ended when the algorithm
enters cases 1(b) or 2(b), and the regret incurred on case 2(a) would
be bounded using the above Lemma~\ref{lemma:functionflat}.

The key idea for this bound is present in Section 9.2.2 of Nemirovski
and Yudin~\cite{NemirovskiYu83}. We need a slight modification of
their argument due to the fact that the function evaluations have
noise and our sampling strategy is a little different from theirs. 

\begin{lemma}
  At any round, the number of visits to case 1(a) is $2d^2\log
  d/\phiconst^2$, and each pyramid $\pyramid$ constructed by the
  algorithm satisfies $\|\apex - \centerpt\| \geq
  \inrad_{\epcount}/d$, where $\apex$ is the apex of $\pyramid$.
  \label{lemma:numpyramidbound}
\end{lemma}

\begin{proof}
  The proof follows by a simple geometric argument that exploits the
  fact that we have an angle $2\wideangle$ at the apex of our pyramid
  which is almost equal to $\pi$, and that $\apex - \centerpt$ and
  $\base_i - \centerpt$ are orthogonal for any pyramid $\pyramid$ we
  construct (see Figure~\ref{fig:pyramid}). By definition of case 1(a),
  $\top \ne \apex$, so we assume $\top = \base_1$ wlog. By
  construction,
  
  \begin{equation}
    \|\base_1 - \centerpt\| = \sin \wideangle\|\apex -
    \centerpt\|. 
    \label{eqn:onestepdistbound}
  \end{equation}
  Since this step applies every time we enter case 1(a), the total
  number $k$ of visits to case 1(a) satisfies
  \begin{equation*}
    \|\base_1 - \centerpt\| = (\sin \wideangle)^k \inrad_{\epcount}, 
  \end{equation*}
  where we recall that $\inrad_{\epcount}$ is the radius of the
  regular simplex we construct in the first step on every round. We
  further note that for a regular simplex of radius
  $\inrad_{\epcount}$, a Euclidean ball of radius
  $\inrad_{\epcount}/d$ is contained in the simplex. We also note that
  by construction, $\cos \wideangle = \phiconst/d$ and hence $\sin
  \wideangle = \sqrt{1-\phiconst^2/d^2} \leq 1 -
  \phiconst^2/(2d^2)$. Hence, setting $k = 2d^2\log d/\phiconst^2$
  suffices to ensure that $\|\base_1 - \centerpt\| \leq
  \inrad_{\epcount}/d$ guaranteeing that $\base_1$ lies in the initial
  simplex of radius $\inrad_{\epcount}$ centered at $\centerpt$, as
  depicted in Figure~\ref{fig:staircase}.

  \begin{figure}
    \centering
    \includegraphics[height=2in]{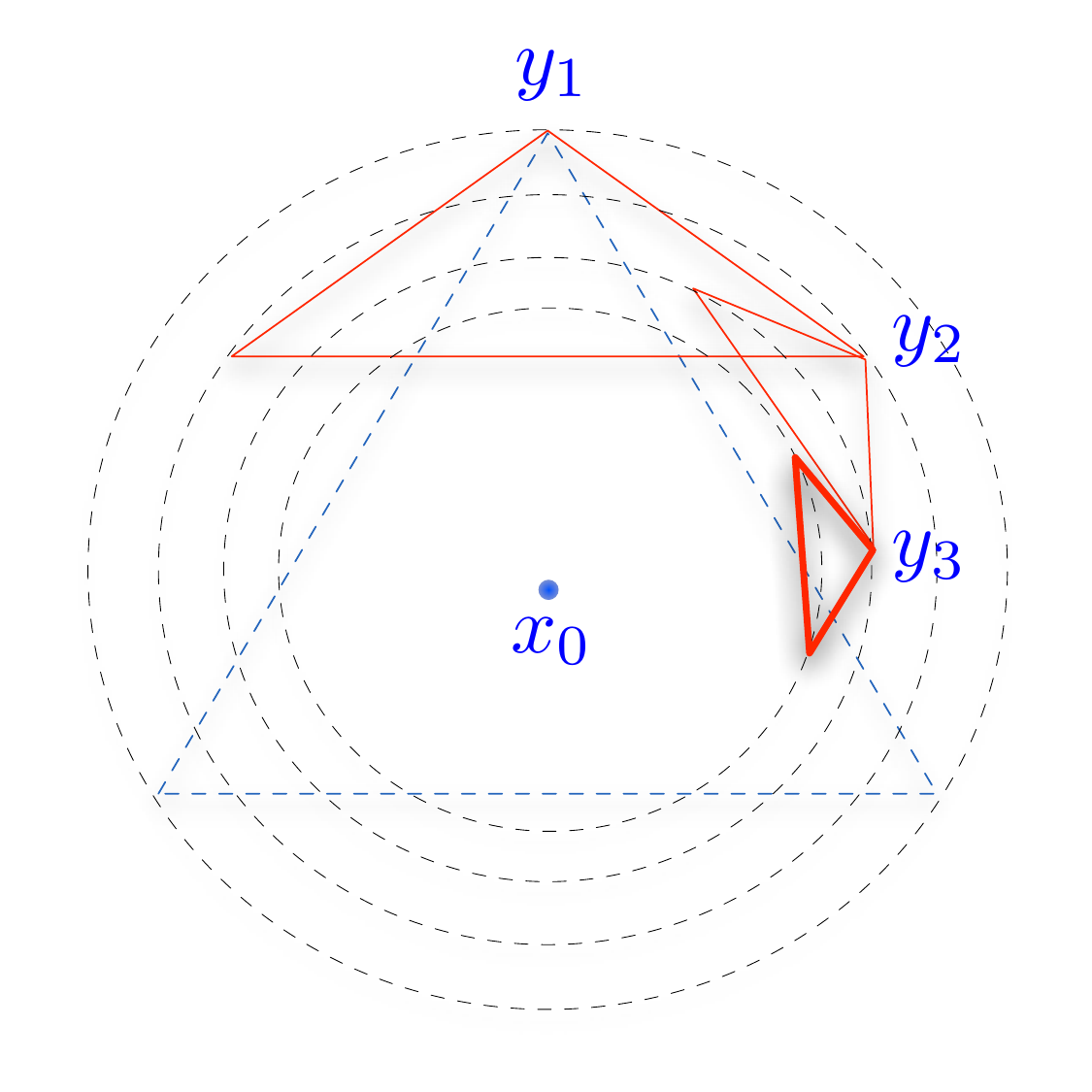}
    \caption{The apexes of the successive pyramids get closer to the
      center of the simplex $\centerpt$ and eventually enter the
      simplex after at most $\order(d^2 \log d)$ pyramids.}
    \label{fig:staircase}
  \end{figure}
  Let $\apex_1,\dots,\apex_k$ be the apexes of the pyramids we have
  constructed in this round. Then by construction, we have a sequence
  of points such that
  
  \begin{equation*}
    f(\base_1) = f(\top) \geq f(\apex_k) + \gamma \geq  f(\apex_{k-1})
    + 2\gamma \dots\geq f(\apex_1) + k\gamma. 
  \end{equation*}
  On the other hand, we know that $\apex_1$ satisfies $f(\apex_1)
  \geq f(x_i) - \gamma$ for all the vertices $x_i$ of the simplex by
  definition of $\apex_1$. Since $\base_1$ lies in the simplex,
  convexity of $f$ guarantees that
  
  \begin{equation*}
    f(\apex_1) \geq f(\base_1) - \gamma \geq f(\apex_1) + (k-1)\gamma ,
  \end{equation*}
  which is a contradiction unless $k \leq 1$. Thus it must be the case
  that $\base_1$ is not in the simplex if $k > 1$, in which case $k$
  can be at most $2d^2\log d/\phiconst^2$. 
\end{proof}

This lemma guarantees that in at most $2d^2\log d/\phiconst^2$ pyramid
constructions, the algorithm will enter one of cases 1(b) or 2(b) and
terminate the epoch, unless the CI level $\gamma$ at this round is
insufficient to resolve things and we end in case 2(a). It also shows
that all the pyramids constructed by our algorithm are sufficiently
far from the center which is assumed by
Lemmas~\ref{lemma:correct1b}-~\ref{lemma:functionflat}. Until now, we
have focused on controlling the regret on the pyramids we construct,
which is convenient since we sample the center points of the
pyramids. To bound the regret incurred over one round, we also need to
control the regret over the initial simplex we query at every
round. We start with a lemma that shows how to control the net regret
accrued over an entire round, when the round ends in case 2(a).

\begin{lemma}
  For any round with a CI width of $\gamma$ that terminates in case
  2(a), the net regret incurred on the round is at most
  \begin{equation*}
    \frac{24d\sigma\log T}{\gamma} \left(\frac{2d^2\log
      d}{\phiconst^2} + 1\right)
    \left(\frac{4d^7\inconst}{\phiconst^3} + \frac{d(d+2)}{\phiconst}
    \right) \left(\frac{12\inconst d^4}{\phiconst^2} + 11\right)
  \end{equation*}
  \label{lemma:flatroundbound}
\end{lemma}

\begin{proof}
  Suppose we constructed a total of $k$ pyramids on the round, with $k
  \leq 2d^2\log d/\phiconst$ by
  Lemma~\ref{lemma:numpyramidbound}. Then we know that the
  instantaneous regret on any point of the $k_{th}$ pyramid
  $\pyramid_k$ is bounded by

  $$\delta \defeq \gamma\left(\frac{4d^7\inconst}{\phiconst^3} +
  \frac{d(d+2)}{\phiconst} \right) \left(\frac{12\inconst
    d^4}{\phiconst^2} + 11\right), $$  
  by Lemma~\ref{lemma:functionflat}. We also note that by
  construction, $\apex_k$ is the $\top$ vertex of the $(k-1)$st
  pyramid $\pyramid_{k-1}$. Hence by definition of case 1(a) (which
  caused us to go from $\pyramid_{k-1}$ to $\pyramid_k$), we know that
  $f(x) \leq f(\apex_k) + \gamma$ for all $x \in
  \pyramid_{k-1}$. Reasoning in the same way, we get that the function
  value at each vertex of the pyramid we constructed in this round is
  bounded by the function value at $\apex_k$. Furthermore, just like
  the proof of Lemma~\ref{lemma:numpyramidbound}, the function value
  at any vertex of the initial simplex is also bounded by the function
  value at $\apex_k$. As a result, the instantaneous regret incurred
  at any point we sampled in this round is bounded by the net regret
  at $\apex_k$ which is at most by $\delta$ using
  Lemma~\ref{lemma:functionflat}. Since every pyramid as well as the
  simplex samples at most $d+2$ vertices, and the total number of
  pyramids we construct is bounded by
  Lemma~\ref{lemma:numpyramidbound}, we query at most
  $(d+2)(2d^2/\phiconst^2\log d + 1)$ points at any round. In order to
  bound the number of queries made at any point, we observe that for a
  CI level $\hatgamma$, we make $2\sigma\log T/\hatgamma^2$
  queries. Suppose $\gamma = 2^{-1}$. Since $\hatgamma$ is
  geometrically decreased to $\gamma$, the total number of queries
  made at any point is bounded by
  \begin{align*}
    \sum_{j=1}^i \frac{2\sigma\log T}{2^{-2j}} \leq 8\sigma\log T
    2^{2i} = \frac{8\sigma\log T}{\gamma^2}.
  \end{align*}
  Putting all the pieces together, the net regret accrued over this
  round is at most
  \begin{equation*}
    \frac{24d\sigma\log T}{\gamma} \left(\frac{2d^2\log
      d}{\phiconst^2} + 1\right)
    \left(\frac{4d^7\inconst}{\phiconst^3} + \frac{d(d+2)}{\phiconst}
    \right) \left(\frac{12\inconst d^4}{\phiconst^2} + 11\right),
  \end{equation*}
  which completes the proof.
\end{proof}

We are now in a position to state a regret bound on the net regret
incurred in any round. The key idea would be to use the bound from
Lemma~\ref{lemma:flatroundbound} to bound the regret even when the
algorithm terminates in cases 1(b) or 2(b).

\begin{lemma}
  For any round that terminates in a CI level $\gamma$, the net regret
  over the round is bounded by
  \begin{equation*}
    \frac{48d\sigma\log T}{\gamma} \left(\frac{2d^2\log
      d}{\phiconst^2} + 1\right)\left(\frac{4d^7\inconst}{\phiconst^3}
    + \frac{d(d+2)}{\phiconst} \right) \left(\frac{12\inconst
      d^4}{\phiconst^2} + 11\right).
  \end{equation*}
  \label{lemma:roundbound}
\end{lemma}

\begin{proof}
  We just need to control the regret incurred in rounds that end in
  cases 1(b) or 2(b). We recall from the description of the algorithm
  that a CI level of $\gamma$ is used at a round only when the
  algorithm terminates the round with a CI level of $2\gamma$ in case
  2(a). The only exception is the first round with $\gamma = 1$, where
  the instantaneous regret is bounded by 1 at any point using the
  Lipschitz assumption. Now suppose we did end a round with CI level
  $2\gamma$ in case 2(a). In particular, the proof of
  Lemma~\ref{lemma:flatroundbound} guarantees that the instantaneous
  regret at any vertex of the simplex we construct is at most
  \begin{equation*}
    2\gamma\left(\frac{4d^7\inconst}{\phiconst^3} +
    \frac{d(d+2)}{\phiconst} \right) \left(\frac{12\inconst
      d^4}{\phiconst^2} + 11\right) 
  \end{equation*}
  
  Now consider any pyramid constructed on this round. We know that the
  instantaneous regret incurred if the pyramid ends in case 2(a) is
  bounded by Lemma~\ref{lemma:functionflat}. Furthermore, if the
  algorithm was in cases 1(a), 1(b) or 2(b) with a CI level
  $\hatgamma$ (which could be larger than $\gamma$ in general), then
  it must have been in case 2(a) with a CI level $2\hatgamma$. Hence
  the instantaneous regret on the vertices of the pyramid is at most
  \begin{equation*}
    2\hatgamma\left(\frac{4d^7\inconst}{\phiconst^3} +
    \frac{d(d+2)}{\phiconst} \right) \left(\frac{12\inconst
      d^4}{\phiconst^2} + 11\right),
  \end{equation*}
  and we make at most $\frac{8\sigma\log T}{\hatgamma^2}$ queries on
  any point of the pyramid by a similar argument like the previous
  lemma. Thus the net regret incurred at any pyramid constructed by
  the algorithm is at most
  \begin{equation*}
    \frac{48d\sigma\log
      T}{\hatgamma}\left(\frac{4d^7\inconst}{\phiconst^3} +
    \frac{d(d+2)}{\phiconst} \right) \left(\frac{12\inconst
      d^4}{\phiconst^2} + 11\right),
  \end{equation*}
  Recalling our bound on the number of pyramids constructed at any
  round completes the proof. 
\end{proof}
Putting all the pieces together, we have shown that the regret
incurred on any round with a CI level $\gamma$ is bounded by
$C/\gamma$, where $C$ comes from the above lemmas. We further observe
that since $\gamma$ is reduced geometrically, the net regret incurred
on an epoch where the largest CI level we encounter is $\gamma$ is
at most 

\begin{equation*}
  \sum_{j=1}^{i}\frac{C}{2^{-j}} \leq 2C2^i = 2C/\gamma.
\end{equation*}
This allows us to get a bound on the regret of one epoch stated in the
next lemma. 
\begin{lemma}
\label{lemma:epochregretbound}
The regret in any epoch which ends in CI level $\gamma$ is at most 

\begin{equation}
  \frac{96d\sigma\log T}{\gamma} \left(\frac{2d^2\log
    d}{\phiconst^2} + 1\right)\left(\frac{4d^7\inconst}{\phiconst^3}
  + \frac{d(d+2)}{\phiconst} \right) \left(\frac{12\inconst
    d^4}{\phiconst^2} + 11\right).
  \label{eqn:epochregretbound}
\end{equation}
\end{lemma}

\subsubsection{Bound on the number of epochs}

In order to bound the number of epochs, we first need to show that the
cone-cutting step discards a sizeable chunk of the set
$\xset_{\epcount}$ in epoch $\epcount$. Recall that we need to
understand the ratio of the volumes of $\Ball_{\epcount+1}$ to
$\Ball_{\epcount}$ in order to understand the amount of volume
discarded in any epoch. 

\begin{lemma}
  Let $\Ball_{\epcount}$ be the smallest ball containing
  $\xset_{\epcount}$, and let $\Ball^{'}_{\epcount+1}$ be the minimum
  volume ellipsoid containing
  $\Ball_{\epcount}\setminus\cone_{\epcount}$. Then for small enough
  constants $\inconst, \phiconst$, $\vol(\Ball^{'}_{\epcount+1}) \leq
  \halving\cdot\vol(\Ball_{\epcount})$ for $\halving =
  \exp(-\frac{1}{4(d+1)})$.
  \label{lemma:volreduction}
\end{lemma}

\begin{proof}
  This lemma is analogous to the volume reduction results proved in
  the analysis of ellipsoid method for convex programming with a
  gradient oracle. We start by arguing that it suffices to consider
  the intersection of $\Ball_{\epcount}$ with a half-space in order to
  understand the set $\Ball_{\epcount}\setminus\cone_{\epcount}$. It
  is clear from the figure that we only increase the volume of the
  enclosing ellipsoid $\Ball^{'}_{\epcount+1}$ if we consider
  discarding only the spherical cap instead of discarding the entire
  cone. But the spherical cap is exactly obtained by taking the
  intersection of $\Ball_{\epcount}$ with a half-space.

  The choices of the constants $\inconst, \phiconst$ earlier
  guarantee that the distance of the hyperplane from the origin is at
  most $\outrad_{\epcount}/(4 (d+1))$. This is because the apex of the
  cone $\cone_{\epcount}$ is always contained in
  $\ball(\inrad_{\epcount})$ by construction and the height of the
  cone is at most $\outrad_{\epcount}\cos\wideanglefinal \leq
  \outrad_{\epcount}/(8(d+1))$ where the last inequality will be
  ensured by construction. Ensuring $\inrad_{\epcount} \leq
  \outrad_{\epcount}/(32(d+1))$ suffices to ensure that the distance
  of the hyperplane to the origin is at most $\outrad_{\epcount}/(4
  (d+1))$.

  Thus $\Ball^{'}_{\epcount+1}$ is the minimum volume ellipsoid
  enclosing the intersection of a sphere with a hyperplane at a
  distance at most $\outrad_{\epcount}/(4 (d+1))$ from its center. The
  volume of $\Ball^{'}_{\epcount+1}$ is then bounded as stated by
  using Theorem 2.1 of Goldfarb and Todd~\cite{GoldfarbTodd82} in
  their work on deep cuts for the ellipsoid algorithm. In particular,
  we apply their result with $\alpha = -1/(4(d+1))$ giving the
  statement of our lemma. 
\end{proof}

We note that the connection from volume reduction to a bound on the
number of epochs is somewhat delicate for our algorithm. The key idea
is to show that at any epoch that ends with a CI level $\gamma$, the
cone $\cone_{\epcount}$ contains points with regret at least
$\gamma$. This will be shown in the next lemma. 

\begin{lemma}
  At any epoch ending with CI level $\gamma$, the instantaneous regret
  of any point in $\cone_{\epcount}$ is at least $\gamma$
  \label{lemma:lowerboundregret}
\end{lemma}

\begin{proof}
Since every epoch terminates either through case 1(b) or through the
case 2(b) followed by hat-raising, we just need to check the condition
of the lemma for both the cases. If the epoch proceeds to cone-cutting
through case 1(b), this is already shown in
Equation~\eqref{eqn:regret1b}. Thus we only need to verify the claim
when we terminate via the hat-raising step. Recall that after
hat-raising, the apex $\apex'$ of the final pyramid $\pyramid'$
constructed in the hat-raising step satisfies that $f(\apex') \geq
f(\base_i) + \gamma$ for all the vertices $\base_1,\dots,\base_d$ of
the pyramid. Consider any point $x \in \cone_{\epcount}$. This point
lies on a ray from the base of $\pyramid'$ passing through
$\apex'$. We know the function $f$ is increasing along this ray at
$\apex'$ and hence continues to increase from $\apex'$ to $x$ by
convexity of $f$, as argued in the proof of
Lemma~\ref{lemma:correcthatraising}. Hence in this case also the
instantaneous regret of any point in $\cone_{\epcount}$ is at least
$\gamma$ completing the proof.
\end{proof}

The above lemma allows us to bound the number of epochs played by the
algorithm. 

\begin{lemma}
  The total number of epochs in the algorithm is bounded by
  $\frac{d\log T}{\log(1/\halving)}$ with $\halving =
  \exp\left(-\frac{1}{4(d+1)}\right)$.
  \label{lemma:numepochbound}
\end{lemma}

\begin{proof}
Let $x^*$ be the optimum of $f$. Since $f$ is 1-Lipschitz, any point
in a ball of radius $1/\sqrt{T}$ centered around $x^*$ has
instantaneous regret at most $1/\sqrt{T}$. The volume of this ball is
$T^{-d/2}\unitvol_d$, where $\unitvol_d$ is the volume of a unit ball
in $d$-dimensions. Suppose the algorithm goes on for $k$ epochs. We
know that the volume of $\xset$ after $k$ epochs is at most
$\halving^k\unitvol_d$ by Lemma~\ref{lemma:volreduction}. We also note
that the instantaneous regret of any point discarded by the algorithm
in any epoch is at least $1/\sqrt{T}$ using
Lemma~\ref{lemma:lowerboundregret}, since we always maintain $\gamma
\geq 1/\sqrt{T}$. Thus any point in the ball of radius $1/\sqrt{T}$
around $x^*$ is never discarded by the algorithm. As a result, the
algorithm must stop once we have
\begin{align*}
  \halving^k\unitvol_d \leq T^{-d/2}\unitvol_d,
\end{align*}
which means $k \leq d\log T/\log 1/\halving$ as claimed. 
\end{proof}

We are now in a position to put together all the pieces. 

\begin{proof}[\mbox{\textbf{Proof of Theorem~\ref{thm:highd}}}]
We are guaranteed that there are at most $d\log T/\log (1/\halving)$
epochs where the regret on each epoch is bounded by
Equation~\ref{eqn:epochregretbound}. Observing that $\gamma \geq
1/\sqrt{T}$ guarantees that every epoch has regret at most
\begin{equation*}
  96d\sigma\sqrt{T}\log T \left(\frac{2d^2\log d}{\phiconst^2} +
  1\right)\left(\frac{4d^7\inconst}{\phiconst^3} +
  \frac{d(d+2)}{\phiconst} \right) \left(\frac{12\inconst
    d^4}{\phiconst^2} + 11\right).
\end{equation*}
Combining with the above bound
on the number of epochs guarantees that the cumulative regret of our
algorithm is bounded by 
\begin{equation*}
  \frac{96d^2\sigma\sqrt{T}\log^2 T}{\log (1/\halving)}
  \left(\frac{2d^2\log d}{\phiconst^2} +
  1\right)\left(\frac{4d^7\inconst}{\phiconst^3} +
  \frac{d(d+2)}{\phiconst} \right) \left(\frac{12\inconst
    d^4}{\phiconst^2} + 11\right).
\end{equation*}

Finally, we recall that the entire analysis this far has been
conditioned on the even $\event$ which assumes that the function value
lies in the confidence intervals we construct at every round. By
design, just like the proof of Theorem~\ref{theorem:regret-1d},
$\P(\event^c) \leq 1/T$. Using this and substituting the value of
$\halving$ from Lemma~\ref{lemma:numepochbound} completes the proof of
the theorem.
\end{proof}

\section{Discussion}

This paper presents a new algorithm for convex optimization when only
noisy function evaluations are possible. The algorithm builds on the
techniques of Nemirovski and Yudin~\cite{NemirovskiYu83} from zeroth
order optimization. The key contribution of our work is to extend
their algorithm to a noisy setting in such a way that a low regret on
the sequence of points queried can be guaranteed. The new algorithm
crucially relies on a \emph{center-point device} that demonstrates the
key differences between a regret minimization and an optimization
guarantee. Our algorithm has the optimal $\order(\sqrt{T})$ scaling of
regret up to logarithmic factors. However, our regret guarantee has a
rather large dimension dependence. As remarked after
Theorem~\ref{thm:highd}, this is unsurprising since the algorithm of
Nemirovski and Yudin~\cite{NemirovskiYu83} has a large dimension dependence even in a
noiseless case. Random walk approaches~\cite{Bertsimas2004} have been
successful to improve the dimension scaling in the noiseless case, and
investigating them for the noisy scenario is an interesting question
for future research.

\subsection*{Acknowledgments}

Part of this work was done while AA and DH were at the University of
Pennsylvania.  AA was partially supported by MSR and Google PhD
fellowships while this work was done.  DH was partially supported
under grants AFOSR FA9550-09-1-0425, NSF IIS-1016061, and NSF
IIS-713540.  AR gratefully acknowledges the support of NSF under grant
CAREER DMS-0954737.

\bibliographystyle{plain}
\bibliography{bib}
\appendix
\section{Properties of pyramid constructions}

We outline some properties of the pyramid construction in this
appendix. Recall that $\wideangle = \arccos(\phiconst/d)$.  For
simplicity, we assume $d \geq 2$.  In this case, $\cos(\wideangle) =
\phiconst/d$ and $\sin(\wideangle) = \sqrt{1-\phiconst^2/d^2} \geq
\cos(\wideangle)$.  Also recall that in epoch $\epcount$, the initial
simplex is contained in $\ball(\inrad_\epcount)$ where
$\inrad_\epcount = \outrad_\epcount / (\inconst d)$.

\begin{lemma} \label{lemma:pyramid-construction}
Let $\pyramid_k$ be the $k$-th pyramid constructed in any round of epoch
$\epcount$.
\begin{enumerate}
\item The distance from the center of $\ball(\inrad_\epcount)$ to the apex
of $\pyramid_k$ is $\inrad_\epcount \sin^{k-1}(\wideangle)$.

\item The distance from the apex of $\pyramid_k$ to any vertex of the base
of $\pyramid_k$ is $\inrad_\epcount \sin^{k-1}(\wideangle)
\cos(\wideangle)$.

\item The height of $\pyramid_k$ (distance of the apex from the base) is
$\inrad_\epcount \sin^{k-1}(\wideangle) \cos^2(\wideangle)$.

\end{enumerate}
\end{lemma}
\begin{proof}
The proof is by induction on $k$.
Let $x_0$ be the center of $\ball(\inrad_\epcount)$, $y_1$ be the apex of
$\pyramid_1$, and $z_1$ be any vertex on the base of $\pyramid_1$.
By construction, $y_1-z_1$ is perpendicular to $z_1-x_0$, so we have
$\|y_1-x_0\| = \inrad_\epcount$,
$\|y_1-z_1\| = \inrad_\epcount \cos(\wideangle)$, and
$\|z_1-x_0\| = \inrad_\epcount \cos(\wideangle)$.
Let $p_1$ be the projection of $y_1$ onto the base of $\pyramid_1$.
The triangle with vertices $y_1,z_1,x_0$ is similar to the triangle with
vertices $y_1,p_1,z_1$.
Therefore $\|y_1-p_1\|$, the height of $\pyramid_1$, is $\inrad_\epcount
\cos^2(\wideangle)$.
This gives the base case of the induction (see
Figure~\ref{fig:pyramid-construction}).

The inductive step follows by noting that the apex of $\pyramid_k$ is a
vertex on the base of $\pyramid_{k-1}$, and therefore the distances scale
as claimed.
\end{proof}

\begin{figure}
\begin{center}
\begin{tikzpicture}
  \foreach \r/\p in { 5.0 / 60 }
    {
      \draw[dashed,gray] (0,0) -- node[black,right] {$\inrad_\epcount$} ($\r*(0,1)$);
      \draw[dashed,gray] (0,0) -- node[black,sloped,below] {$\inrad_\epcount\sin(\wideangle)$} ($\r*sin(\p)*cos(\p)*(-1,0)+\r*sin(\p)*sin(\p)*(0,1)$);

      \draw ($(0,\r)$) -- node[black,sloped,above] {$\inrad_\epcount\cos(\wideangle)$} ($\r*sin(\p)*cos(\p)*(-1,0)+\r*sin(\p)*sin(\p)*(0,1)$);
      \draw ($(0,\r)$) -- ($\r*sin(\p)*cos(\p)*(1,0)+\r*sin(\p)*sin(\p)*(0,1)$);
      \draw ($\r*sin(\p)*cos(\p)*(-1,0)+\r*sin(\p)*sin(\p)*(0,1)$) -- ($\r*sin(\p)*cos(\p)*(1,0)+\r*sin(\p)*sin(\p)*(0,1)$);

      \filldraw [black] (0,0) circle (1pt) node[anchor=north] {$x_0$};
      \filldraw [black] ($\r*(0,1)$) circle (1pt) node[anchor=south] {$y_1$};
      \filldraw [black] ($\r*sin(\p)*cos(\p)*(-1,0)+\r*sin(\p)*sin(\p)*(0,1)$) circle (1pt) node[anchor=east] {$z_1$};
      \filldraw [black] ($\r*sin(\p)*cos(\p)*(1,0)+\r*sin(\p)*sin(\p)*(0,1)$) circle (1pt);
      \filldraw [black] ($\r*sin(\p)*sin(\p)*(0,1)$) circle (1pt) node[anchor=north east] {$p_1$};
    }
\end{tikzpicture}
\end{center}
\caption{Construction of pyramids.}
\label{fig:pyramid-construction}
\end{figure}
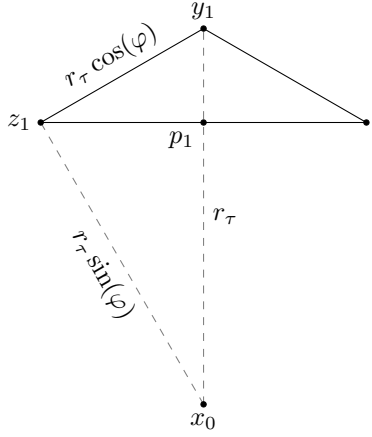

\begin{lemma} \label{lemma:pyramid-inscribed}
Let $\pyramid$ be any pyramid constructed in epoch $\epcount$ with apex at
distance $r_\pyramid \geq \inrad_\epcount / d$ from the center of
$\ball(\inrad_\epcount)$.
Let $\ball_\pyramid$ be the largest ball in $\pyramid$ centered at the
center of mass $c$ of $\pyramid$.
\begin{enumerate}
\item $\ball_\pyramid$ has radius at least $r_\pyramid
\cos^2(\wideangle)/(d+1) \geq \inrad_\epcount c_2^2/(2d^4)$.

\item Let $x \in \pyramid$, and let $b \in \pyramid$ be the point on the
face of $\pyramid$ such that $c = \alpha x + (1-\alpha) b$ for some $0 <
\alpha \leq 1$.
Then $(1-\alpha)/\alpha \leq (d+1)d/c_2$.
\end{enumerate}

\end{lemma}
\begin{proof}
Let $h$ be the height of $\pyramid$.
By Lemma~\ref{lemma:pyramid-construction}, $h = r_\pyramid
\cos^2(\wideangle)$.
The distance from $c$ to the base of $\pyramid$ is
\[ \frac{h}{d+1} = \frac{r_\pyramid \cos^2(\wideangle)}{d+1} , \]
and the distance from $c$ to any other face of $\pyramid$ is
\[ \sin(\wideangle) \left(1 - \frac1{d+1}\right) h
= \sqrt{1 - \cos^2(\wideangle)}
\left(1 - \frac1{d+1}\right)
r_\pyramid \cos^2(\wideangle)
\geq \frac{r_\pyramid \cos^2(\wideangle)}{2}
\]
(here we have used $d \geq 2$ and $\cos(\wideangle)\leq1/d$).
Therefore $\ball_\pyramid$ has radius at least
\[
\frac{r_\pyramid \cos^2(\wideangle)}{d+1}
\geq \frac{\inrad_\epcount}{d} \cdot \frac{c_2^2/d^2}{d+1}
= \frac{\inrad_\epcount c_2^2}{d^3(d+1)}
\geq \frac{\inrad_\epcount c_2^2}{2d^4}
.
\]
which proves the first claim.

For the second claim, note that $\alpha = \|b-c\| / (\|b-c\| + \|x-c\|)$;
moreover, $\|b-c\|$ is at least the radius of $\ball_\pyramid$, and
$\|x-c\|$ is at most the distance from $c$ to any vertex of $\pyramid$.
By Lemma~\ref{lemma:pyramid-construction}, the distance from $c$ to a
vertex on the base of $\pyramid$ is
\[
\sqrt{\left( \frac{r_\pyramid}{d+1} \cos^2(\wideangle) \right)^2
+ \left( r_\pyramid \cos(\wideangle) \sin(\wideangle) \right)^2}
= \frac{r_\pyramid \cos^2(\wideangle)}{d+1} \sqrt{ 1 +
\frac{(d+1)^2\sin^2(\wideangle)}{\cos^2(\wideangle)} }
\]
and the distance from $c$ to the apex of $\pyramid$ is
\[
\left(1 - \frac1{d+1}\right) h
= \left(1 - \frac1{d+1}\right) r_\pyramid \cos^2(\wideangle)
= \frac{d}{d+1} r_\pyramid \cos^2(\wideangle)
.
\]
Therefore, by the first claim and Lemma~\ref{lemma:pyramid-construction},
\begin{align*}
\frac{1-\alpha}{\alpha}
= \frac{\|x-c\|}{\|b-c\|}
& \leq \max\left\{
\frac{\frac{d r_\pyramid \cos^2(\wideangle)}{d+1}}
{\frac{r_\pyramid \cos^2(\wideangle)}{d+1}}
, \
\frac{\frac{r_\pyramid \cos^2(\wideangle)}{d+1} \sqrt{ 1 +
\frac{(d+1)^2\sin^2(\wideangle)}{\cos^2(\wideangle)} }}
{\frac{r_\pyramid \cos^2(\wideangle)}{d+1}}
\right\}
\\
& =
\max\left\{ d, \ \sqrt{1 + (d+1)^2 \left(\frac1{\cos^2(\wideangle)} - 1
\right)} \right\}
\\
& \leq
\max\left\{ d, \ \sqrt{\frac{(d+1)^2}{\cos^2(\wideangle)}} \right\}
\\
& =
\max\left\{ d, \ \frac{d+1}{\cos(\wideangle)} \right\}
\\
& =
\max\left\{ d, \ \frac{(d+1)d}{c_2} \right\}
\\
& = \frac{(d+1)d}{c_2}
.
\end{align*}
\end{proof}

\end{document}